\documentclass[11pt]{amsart}
\usepackage[utf8]{inputenc}
\usepackage{amsmath,amsfonts,amssymb,amscd}
\usepackage{tikz-cd} 
\usepackage{xcolor}
\usepackage{float}
\usepackage[abs]{overpic}
\usepackage{hyperref}
\usepackage{mathrsfs}
\usepackage{graphicx}
\usepackage[normalem]{ulem}
\usepackage[nodisplayskipstretch]{setspace}
\hypersetup{colorlinks=true, citecolor=blue, linkcolor=red}

\newcommand{\rl}{{\mathbb{R}}}

\newcommand{\codim}{{\rm codim\,}}
\newcommand{\coker}{\rm coker\,}

\newcommand{\eps}{\varepsilon}

\newcommand{\tmop}[1]{\ensuremath{\operatorname{#1}}}
\renewcommand{\Re}{\tmop{Re}}
\renewcommand{\Im}{\tmop{Im}}

\newcommand{\C}{\mathbb C}
\newcommand{\calC}{\mathcal C}

\newcommand{\calK}{\mathcal K}
\newcommand{\N}{\mathbb N}

\newcommand{\R}{\mathbb R}
\newcommand{\calR}{\mathcal R}

\newcommand{\Cn}{\mathbb C^n}
\newcommand{\M}{\mathcal M}
\newcommand{\calS}{\mathcal S}
\newcommand{\W}{\mathcal W}
\newcommand{\cont}{\mathcal C}
\newcommand{\wt}{\widetilde}

\newcommand{\beas}{\begin{eqnarray*}}
\newcommand{\eeas}{\end{eqnarray*}}
\newcommand{\bes} {\begin{equation*}}
\newcommand{\ees} {\end{equation*}}
\newcommand{\be} {\begin{equation}}
\newcommand{\ee} {\end{equation}}
\newcommand{\bea} {\begin{eqnarray}}
\newcommand{\eea} {\end{eqnarray}}

\newcommand\partl[2]{\dfrac{\partial{#1}}{\partial{#2}}}

\newtheorem{theorem}{Theorem}
\newtheorem{lemma}[theorem]{Lemma}
\newtheorem{prop}[theorem]{Proposition}

\theoremstyle{definition}
\newtheorem{defn}[theorem]{Definition}
\newtheorem{rem}[theorem]{Remark}

\makeatletter
\renewcommand*\env@matrix[1][\arraystretch]{%
  \edef\arraystretch{#1}%
  \hskip -\arraycolsep
  \let\@ifnextchar\new@ifnextchar
  \array{*\c@MaxMatrixCols c}}
\makeatother
\begin{document}
\title{Polynomially convex embeddings and CR singularities of real manifolds} 
\date{\today}
\author{Purvi Gupta}
\address{Department of Mathematics, Indian Institute of Science, Bangalore 560012, India}
\email{purvigupta@iisc.ac.in}
\author{Rasul Shafikov}
\address{Department of Mathematics, University of Western Ontario,  London, Ontario N6A 5B7, Canada}
\email{shafikov@uwo.ca}

\begin{abstract}
It is proved that any smooth manifold $\M$ of dimension $m$ admits a smooth polynomially convex embedding into $\C^n$ when $n \ge \lfloor 5m/4\rfloor$. Further, such embeddings are dense in the space of 
smooth maps from $\M$ into $\C^n$ in the $\cont^3$-topology. The components of any such embedding give smooth generators of the algebra of complex-valued continuous functions on $\M$. A key ingredient of the proof is a coordinate-free description of certain notions of (non)degeneracy, as defined by Webster and Coffman, for CR-singularities of order one of an embedded real manifold in $\Cn$. The main result is obtained by inductively perturbing each stratum of degeneracy to produce a global polynomially convex embedding.
\end{abstract}

\maketitle

\tableofcontents
\section{Introduction}

A common theme in modern geometry and analysis is the existence of embeddings of manifolds into Euclidean spaces that satisfy additional properties. For example, isometric embeddings of Riemannian manifolds, Lagrangian embeddings into $\R^{2n}$ equipped with the standard symplectic structure, proper holomorphic embeddings of complex manifolds into $\C^n$, etc., all fall under this category. In this spirit, we are interested in the polynomially convex embedding problem: {\it find the smallest dimension $n$ such that any closed smooth real $m$-dimensional manifold admits a smooth polynomially convex  embedding into $\C^n$.} 
Recall that a compact set $K\subset \C^n$ is called polynomially convex, if its polynomially convex hull, defined as 
$$
\widehat K = \{z\in \C^n : |P(z) | \le ||p||_K \text{ for all holomorphic polynomials } P\},
$$ 
agrees with $K$, i.e., $K=\widehat K$. Polynomial convexity is important because of its connection with the approximation theory: by the classical Oka-Weil theorem, a holomorphic function in a neighbourhood of $K$ can be uniformly approximated on $K$ by holomorphic  polynomials.

In this formulation of the above embedding problem, the required dimension $n$ depends only on $m$ and not on the choice of a particular $m$-manifold. It is well-known, see, e.g., Stout~\cite{St}, that 
$n>m$ as no closed $n$-dimensional submanifold of $\C^n$ can be polynomially convex, and through totally real embeddings 
one, can take $n=\lfloor \frac{3m}{2}\rfloor$. Recall that a manifold $M\subset \C^n$ is called totally real if for any 
$p\in M$, the tangent space $T_p M $, considered as a linear subspace of $T_p \C^n$, does not contain any complex 
subspaces of positive dimension.
For certain classes of manifolds, such as $m$-spheres, the optimal value is already $n=m+1$. The same holds true for all closed $m$-manifolds, if we consider topological and not smooth embeddings, see Vodovoz-Zaidenberg~\cite{VZ}. For smooth embeddings, a better bound for $n$ 
(than $\lfloor \frac{3m}{2}\rfloor$) was given  in \cite{GS1} and \cite{GS2}.
But it general, it is a difficult problem to determine the optimal $n$ for smooth embeddings of all $m$-manifolds. In this paper, we give a further improvement on the dimension $n$. Our main result is as follows. 

\begin{theorem}\label{t.main}
Let ${\bf M}_m$ the class of all smooth closed (i.e., compact without boundary) manifolds of dimension $m>0$. Let $\mathcal M \in {\bf M}_{m}$. 
Let $n \ge \lfloor 5m/4\rfloor$. 
Then, given any $\calC^\infty$-smooth embedding $f: \mathcal M\hookrightarrow \Cn$, there exists a $\calC^3$-small perturbation of 
$\widetilde f$ of $f$ such that $\widetilde f: \mathcal M\hookrightarrow\Cn$ is a smooth embedding and $\widetilde f(\mathcal M)$ is 
polynomially convex in $\Cn$. If $\tilde f =(f_1,\dots, f_n)$ is a coordinate representation, any continuous complex-valued function on $\M$ can be approximated uniformly on $\M$ by functions of the form $P(f_1, \dots, f_n)$, where $P(z_1,\dots,z_n)$ are holomorphic polynomials.
\end{theorem}

Note that the theorem implies that polynomially convex embeddings of $\M$ are dense in the $\calC^3$-topology 
on the space of smooth maps
$f : \M \to \C^n$. Furthermore, the above result is optimal for $m\leq 7$. 
The approximation-theoretic statement in the theorem gives an upper bound on the size of the minimal set of smooth generators of the continuous algebra on $\M$. Such a statement is one of the primary applications of the polynomially convex emebedding problem, and can be viewed as a Weierstrass-type approximation theorem on $\M$.


If a manifold $\M$ of dimension $m<n$ admits a totally real embedding into $\C^n$, the result follows from the work of Forstneri\v c-Rosay~\cite{FR}, Forstneri{\v c}~\cite{Fo94}, and Low-Wold~\cite{LW}. However, by Ho-Jacobowitz-Landweber~\cite{HJL}, for every $n <\lfloor \frac{3m}{2}\rfloor$ there exists $\M \in {\bf M}_m$ that does not admit a totally real embedding into $\C^n$. Thus, in further considerations, an important role is played by the set of so-called CR-singularities, i.e, the points where the tangent plane to an embedded $\M$ inherits complex directions from the ambient $\C^n$. More precisely, in the context of the embedding problem we consider a real submanifold $M\subset \C^n$ which is totally real on a dense open set. If $p \in M$ is not a totally real point, then $\dim T_p M \cap JT_p M = k >0$, and we call $p$ a CR-singularity of order $k$. Here $J$ is the standard complex structure on $\C^n$. 

In complex dimension 2, CR-singularities, known as complex points, are well-studied. For example, properly counted, the complex points of a real surface $M \subset \C^2$ give the Euler characteristic of $M$, see Lai~\cite{L}. 
In higher dimensions, much less is known. By a standard transversality argument (see Domrin~\cite{D}, Forstneri\v c~\cite{For},  and Proposition~\ref{P:PertI} below) one can show that the set of CR-singularities can be stratified into (nonclosed) manifolds of certain dimensions. Under the assumption of real-analyticity and some nondegeneracy conditions, Beloshapka for $(m,n)=(4,5)$ and Coffman for 
$\frac{2}{3}(n+1) \le m < n$ constructed a normal form, i.e., a model submanifold that is locally biholomorphically 
equivalent to the original manifold, in a neighbourhood of a given CR-singular point of order 1. It was proved in \cite{GS1} that the Coffman model manifolds are locally polynomially convex. At the same time, very little is known about the geometry and convexity of real submanifolds near CR-singularities of higher order. Further, the nondegeneracy conditions that allow for the construction of the Coffman normal form are defined purely locally, in terms of coefficients of second-order terms of the defining functions of the manifold, making it difficult to describe the global behaviour of the locus of degenerate CR-singular points. 

Generalizing the work of Bishop~\cite{Bi}, for suitable immersed manifolds of dimension $m$ in $\C^m$, Webster constructed an invariant function on the set of order 1 CR-singularities whose sublevel sets characterize CR-singularities as hyperbolic, parabolic, and elliptic. Further, by constructing certain vector bundles on the set of CR singularities, Webster defined the parabolic index of CR-singular points. 
In this work, we use Webster's ideas to give a general description and stratification of certain degeneracies of order-one CR-singular points of an embedded manifold $M \subset \C^n$. These are defined in terms of the rank of the projection between certain vector bundles defined on the submanifold $S_1$ of CR-singularities of order 1. In particular, this identifies the set of CR-singular points where the manifold of CR-singularities is itself not totally real. In \cite{C}, Coffman defined a different notion of degeneracy for CR-singularities of order 1, that is relevant for the construction of normal forms. Again, we were able to characterize these degeneracies in terms of certain globally defined  vector bundles (on $S_1$). This gives a coordinate-free representation and a global stratification of $S_1$ into manifolds of these degeneracies, whose dimensions depend on the pair $(m,n)$. This is summarized in Proposition~\ref{P:PertII}.

Stratifying the set of CR-singularities into sets of various degeneracies gives us an inductive way of constructing a small perturbation of a given embedding $f: \M \to \C^n$ that is  polynomially convex. In the dimensional setting of our main theorem, all the CR-singularities of a generically embedded $M\subset\Cn$ are of order $1$, and form a totally real submanifold $S_1\subset M$. Furthermore, the set of Coffman degenerate points, $C_1\subset S_1$, is a submanifold of dimension at most one, i.e., is either a finite set or a disjoint union of smooth closed curves. We show that a small perturbation of $f$ makes $C_1$ polynomially convex. Using the result of Arosio-Wold~\cite{ArWo19}, we perturb the remaining part of $S_1$ (while keeping a neighbourhood of $C_1$ in $M$ fixed) to make it polynomially convex. This is matched with an ambient perturbation of $M$, so that the perturbed $S_1$ is realized as the CR-singular set of the perturbed $M$.  
Finally, we perturb the totally real portion of the embedded manifold to make all of $M$ polynomially convex. 

Another fundamental problem in the theory of polynomial convexity is the construction of compact sets whose polynomially 
convex hull is nontrivial but does not contain any analytic structure, i.e., the hull contains no nonconstant analytic discs. 
Examples of such compact sets goes back to the work of Stolzenberg~\cite{S}. In~\cite{IS}, Izzo-Stout raised the question of finding the optimal $n$ such that any compact  $m$-manifold can be smoothly embedded into $\C^n$ as a nonpolynomially convex compact whose hull contains no analytic structure. Note that unlike the polynomially convex embedding problem, where one cannot obtain a bound better than $n=m+1$, no such lower bound exists for embeddings with nontrivial hulls lacking analytic structure. We remark that 
the same argument as in our previous work \cite{GS2} shows that in the dimensional setting of Theorem~\ref{t.main}, there exists an embedding of $\M$ which is not polynomially convex, but whose hull contains no analytic structure.
\smallskip 

\noindent{\bf Organization of the paper.} The paper is organized as follows. In Section~\ref{S:PertI}, we describe a stratification of the set of CR singularities of a generic compact $m$-fold in $\Cn$ in terms of the order of the singularities (Proposition~\ref{P:PertI}). In Section~\ref{S:nondeg}, we focus on CR singularities of order $1$, and define some notions of degeneracy for such points in terms of globally defined geometric objects (independent of local coordinates). The main result in this section (Proposition~\ref{P:LocCoord}) is a description of these notions in terms of local coordinates. In Section~\ref{S:strat}, we describe how the various sets of degeneracy introduced in the previous section stratify the set of order-one CR singularities. The dimensions of these sets (in general position) is described in Proposition~\ref{P:PertII}. 
In Section~\ref{S:pert}, we give a refinement of previously known results on global perturbations of totally real manifolds (off of polynomially convex sets). In Section~\ref{S:C2}, we show that a manifold with only Coffman nondegenerate CR singularities off of a polynomially convex set can be globally perturbed to be polynomially convex. In Section~\ref{S:main}, we complete the proof of the main theorem, by tackling the case where degenerate CR singularities appear. 
\smallskip 

\noindent {\bf Notation.} Throughout the paper, we denote by $\M$ an abstract smooth compact manifold of dimension $m$. If $f: \M \to\C^n$ is a smooth map, then we set $M=f(\M)$. We also follow the convention of denoting objects associated to the abstract $\M$ by calligraphic letters, and their images 
in $\Cn$ by Roman letters. For example, for $\calS\subset \M$, we set $S = f(\calS)$.
\smallskip 

\noindent{\bf Remark.}
On March 26, 2025, after the work on this paper was essentially complete, a paper titled ``The polynomially convex embedding dimension for real manifolds of dimension $\le 11$" by Leandro Arosio, H\r akan Samuelsson Kalm, and Erlend F. Wold, appeared on the ArXiv, see https://arxiv.org/abs/2503.19765. This paper discusses similar results on polynomially convex embeddings of real manifolds in $\C^n$. We remark that this paper is written completely independently, and we have not used 
any results or ideas formulated in the preprint by Arosio, Samuelsson Kalm, and Wold, apart from those that already appeared in the papers by L\o w-Wold~\cite{LW} and Arosio-Wold~\cite{ArWo19} as cited here.
\smallskip 

\noindent {\bf Acknowledgments.}
The authors would like to thank Stefan Nemirovski for suggesting Proposition~\ref{l.pertlemma}, which was the starting point for the work presented in this paper. We would also like to thank Adam Coffman for valuable discussions, and comments on a preliminary version of this paper. The first author is partially supported by the  DST FIST Program - 2021 [TPN700661] and an ANRF MATRICS grant [MTR/2023/000393]. The second author is partially supported by Natural Sciences and Engineering Research Council of Canada.

\section{A stratification of the set of CR singular points}\label{S:PertI}
In this section, a stratification of the set of CR singularities of a generic compact $m$-manifold $M\subset\Cn$ into sets of CR singularities of different orders is given. The main result (Proposition~\ref{P:PertI}) is known in the literature (e.g., see \cite{D}), but we cast it in language that is of use later. 

Throughout this section, $J$ denotes the standard complex structure on $\Cn$. 
Let $f: \M\hookrightarrow \C^n$ be a smooth embedding, $m\leq n$, and $M=f(\M)$. For $q\in M$, let $H_q M$ denote the maximal complex subspace of $T_q M$. For $p\in \M$, let $\mathcal H_p=f^*(H_{f(p)} M)$. Since $H_{f(p)} M$ is $J$-invariant, $\mathcal J=(df^{-1}\circ J\circ df)_p$ gives a complex structure on $\mathcal H_p$. By complexifying $\mathcal H_p$ we extends $\mathcal J$ to $\mathcal J_\C$, and so we obtain the $+i$ and $-i$ eigenspaces of $\mathcal J_\C$ given by 
	\beas
		(\mathcal H_\C)_p&=&\text{span}_\C\{v\otimes 1-\mathcal J(v)\otimes i:v\in \mathcal H_p\}\\
		(\overline {\mathcal H_\C})_p&=&\text{span}_\C\{v\otimes 1+\mathcal J(v)\otimes i:v\in \mathcal H_p\}.
	\eeas

Now, we associate a Gauss map $\Phi_f$ to $f$ as follows. Consider $df:T\M\rightarrow T\C^n|_{M}\cong \mathcal M\times\Cn$. Complexifying $T\M$ extends $df$ to $T^{\C\!}\M=T\M\otimes\C$ as follows:
\beas 
	df_\C:T^{\C\!} \M&\rightarrow &\M\times\C\\
	(p,v\otimes \alpha)&\mapsto&(p,\alpha df_p(v)), \quad v\in T_p \M, \ \alpha\in\C, \ p\in \M. 
	\eeas
In terms of matrix representations, if 
	\bes
		A=\begin{pmatrix}
			a_{11} & \cdots & a_{1m}\\
		c_{11} & \cdots & c_{1m}\\
				\vdots & \ddots & \vdots\\
				 a_{n1} & \cdots & a_{nm}\\
				 c_{n1} & \cdots & c_{nm}
		\end{pmatrix}_{2n\times m}\in\operatorname{Mat}_{2n\times m}(\rl),
	\ees
then
	\bes
		A_\C=\begin{pmatrix}
			a_{11}+ic_{11} & \cdots & a_{1m}+ic_{1m}\\
				\vdots & \ddots & \vdots\\
				 a_{n1}+ic_{n1} & \cdots & a_{nm}+ic_{nm}\\
		\end{pmatrix}_{n\times m}\in \operatorname{Mat}_{n\times m}(\C),
	\ees
where $\operatorname{Mat}_{r\times s}(F)$ denotes the space of $r\times s$ matrices over the field $F$. Note that $\ker (df_\C)(p)=(\overline {\mathcal H_\C})_p$. Define $\Phi_f\in\Gamma(\M,\operatorname{Hom}_\C(T^{\C\!}\M,\M\times\Cn))$ by
	\be\label{E:Gauss map}
		\Phi_f:p\mapsto (df_\C)(p) .
	\ee

Given two complex vector bundles $P$ and $Q$ over $\M$ of rank $m\leq n$, respectively, denote by 
$\operatorname{Hom}_\C^\nu(P,Q)$, $1\leq \nu\leq m$,  the fiber bundle over $\M$ where the fiber over 
$p\in \M$ is the space of $\C$-linear transformations from $P_p$ to $Q_p$ that are of rank $m-\nu$. Then, 
$\Phi_f(p)\in \operatorname{Hom}_\C^\nu(T^{\C\!} \M,\M\times\Cn)$ if and only if $f(p)$ is a CR-singularity of $M$ 
of order $\nu$, $\nu=1,\dots,\lfloor m/2\rfloor$. Let $\calS_\nu\subset \M$ denote the set whose image under $f$ is 
the set of CR-singularities of $M$ of order $\nu$, $\nu=1,...,\lfloor m/2\rfloor$. Let 
$\calS=\cup_{\nu=1}^{\lfloor m/2\rfloor }{\calS_\nu}$. 

Let $C^\infty(\M,\Cn)$ be the space of smooth maps from $\M$ to $\Cn$ equipped with the Whitney topology. A set $A \subset C^\infty(\M,\Cn)$ is called residual if it is a countable intersection of dense open subsets of $C^\infty(\M,\Cn)$, in particular, 
it is itself dense. The following is a standard result in the theory of CR-singularities, see \cite{D}, but we provide a proof that is 
invoked in later computations. 

\begin{prop}\label{P:PertI} 
Given a smooth embedding $f: \M\hookrightarrow\Cn$, let $P=T^{\C\!} \M$ and 
$Q=\M \times \Cn$. The set of smooth embeddings $f$ such that $\Phi_f: \M\rightarrow\operatorname{Hom}_{\C}(P,Q)$ is 
transverse to $\operatorname{Hom}_\C^\nu(P,Q)$ for all $\nu=1,...,\lfloor{m/2}\rfloor$, is a residual set in 
$\calC^\infty(\M,\Cn)$. In particular, any smooth map $f: \M\hookrightarrow\Cn$, after a generic small perturbation, is an 
embedding such that each $\calS_\nu$ is either empty, or a (not necessarily closed) submanifold of real codimension 
$2\nu(n-m+\nu)$ in $\M$, and $\overline{\calS_\nu}=\cup_{j=\nu}^{\lfloor m/2\rfloor} \calS_j$.     
\end{prop}

\begin{proof} 
Let $f: \M\hookrightarrow \Cn$ be any embedding. Let $\Phi_f: \M\rightarrow\operatorname{Hom}_{\C\!}(P,Q)$ be the associated Gauss map as defined in \eqref{E:Gauss map}. We will use the parametric transversality theorem; see \cite[Theorem 2.7]{H}. Let $Y$ denote the set of all affine transformations ${\mathbf A}(z)=A(z)+a$, where $A\in \operatorname{Gl}(2n,\rl)$, $a\in\rl^{2n}$ and $z$ is viewed as an element in $\rl^{2n}$. Then 
$$
({\mathbf A},p)\mapsto \Phi_{{\mathbf A}\circ f}(p)=d({\mathbf A}\circ f)_\C(p)=(A\circ df)_\C(p)
$$ 
defines a smooth map $G:Y\times \M\rightarrow \operatorname{Hom}_{\C\,} (P,Q)$. We claim that $G$ is transverse to $\operatorname{Hom}_{\C}^\nu(P,Q)$ for all $\nu\leq\lfloor m/2\rfloor$. 

Let $G({\mathbf A}, p)\in \operatorname{Hom}_\C^\nu(P,Q)$ for some fixed $\nu\in\{1,...,\lfloor m/2\rfloor\}$. 
Relabelling ${\mathbf A}\circ f$ as $f$, it suffices to check that at $u=(I,p)$, where $I$ is the identity map on $\rl^{2n}$,
	\bes
		T_{G(u)}\operatorname{Hom}_\C^\nu(P,Q)+T_{G(u)}\left(G(Y\times \M)\right)=T_{G(u)}\operatorname{Hom}_{\C\!}\,(P,Q).
	\ees 
Let $N_{df_\C(p)}\operatorname{Hom}_\C^\nu(P_p,Q_p)$ denote the normal space of $\operatorname{Hom}_\C^\nu(P_p,Q_p)$ in $\operatorname{Hom}_{\C\!}(P_p,Q_p)$ at $df_\C(p)$. Then, 
\beas
		T_{G(u)}\operatorname{Hom}_{\C\!}\,(P,Q)&=&
			T_p \M\oplus T_{df_\C(p)}\operatorname{Hom}_\C^\nu(P_p,Q_p)\oplus N_{df_\C(p)}\operatorname{Hom}_\C^\nu(P_p,Q_p)\\
		&=&T_{G(u)}\operatorname{Hom}_\C^\nu(P,Q)\oplus N_{df_\C(p)}\operatorname{Hom}_\C^\nu(P_p,Q_p).
	\eeas
We will show that $N_{df_\C(p)}\operatorname{Hom}_\C^\nu(P_p,Q_p)\subset T_{G(u)}(G(Y\times\{p\}))=DG_u\left(T_{u}(Y\times \{p\})\right)$, which establishes the transversality claim. 

Since we work entirely within the fiber over $p$, we make the following identifications:
\begin{itemize}
\item $T_u(Y\times\{p\})\cong\text{Gl}(2n;\rl)$,
\item $\operatorname{Hom}_{\C\!}(P_p,Q_p)\cong \operatorname{Mat}_{n\times m}(\C)$, and
\item $\operatorname{Hom}_\C^\nu(P_p,W_p)\cong \operatorname{Mat}^\nu_{n\times m}(\C)$,
\end{itemize}
where $\operatorname{Mat}^\nu_{n\times m}(\C)$ denotes the space of $n\times m$ complex matrices of rank $m-\nu$. After a holomorphic change of coordinates, one may assume that 
	\bea\label{E:dfC}
		G(u)=df_\C(p)=	\left(\begin{array}{c | c}
		\begin{array}{c|c}
			\mathbf{I}_{\nu\times\nu} & i\mathbf{I}_{\nu\times\nu} \\
				\hline
			\mathbf{0}_{(m-2\nu)\times\nu} & \mathbf{0}_{(m-2\nu)\times\nu} 
		\end{array}
		& 
		\begin{array}{c}
			\mathbf{0}_{\nu\times(m-2\nu)} \\
				\hline
			\mathbf{I}_{(m-2\nu)\times(m-2\nu)}
		\end{array}\\
		\hline
		  		\mathbf{0}_{(n-m+\nu)\times 2\nu} & \mathbf{0}_{(n-m+\nu)\times(m-2\nu)}
 \end{array}\right),
	\eea
where $\mathbf{0}_{r\times s}$ and $\mathbf{I}_{r\times r}$ denote the $r\times s$ zero matrix and $r\times r$ identity matrix, respectively. Above, we have used the fact that $f$ is an embedding, i.e., $\operatorname{rank}_{\rl\!}df=m$, and $\operatorname{rank}_{\C\!}df_\C=m-\nu$. We use the coordinates $W=(w_{jk})_{j=1,...,n}^{k=1,...,m}$ for a matrix $W\in \operatorname{Mat}_{n\times m}(\C)$. For a fixed $W$, $r\in\{m-\nu+1,...,n\}$ and $s\in\{1,...,\nu\}$, let $W^{[rs]}$ denote the $(m-\nu+1)\times (m-\nu+1)$ matrix obtained from $W$ by   
	\begin{itemize}
	\item omitting all the lower $n-m+\nu$ rows, except for the $r$-th one, and
	\item omitting all the leftmost $\nu$ columns, except for the $s$-th one. 
	\end{itemize}
Then, locally near $G(u)=df_\C(p)$,
\beas 
	\operatorname{Mat}^\nu_{n\times m}(\C)=\left\{W\in\operatorname{Mat}_{n\times m}(\C)
		:  \phi_{rs}(W):=\det W^{[rs]}=0, r=m-\nu+1,...,n, s=1,...,\nu \right\},\\
		N_{G(u)}\operatorname{Mat}^\nu_{n\times m}(\C)=\text{span}_{\C\!}\left\{E^{[rs]}=\left(\overline{\partl{\phi_{rs}}{ w_{jk}}}(G(u))\right)_{j=1,...,n}^{k=1,...,m}:r=m-\nu+1,...,n, s=1,...,\nu\right\}.
\eeas
We compute the matrix $E^{[rs]}$ for $r=m-\nu+1$ and $s=1$. The computation goes along similar lines for all other values of $r$ and $s$. For a fixed $W=(w_{jk})$, let $W^{[(m-\nu+1)1]}=(\wt w_{jk})_{j,k=1,...,m-\nu+1}$. Then, 
for $j=1,...,m-\nu+1$, 
\beas
	\wt w_{jk}=\begin{cases}
		w_{j1},& \text{if }k=1,\\
		w_{j(k+\nu-1)},& \text{if }k=2,...,m-\nu+1.
	\end{cases}
\eeas
Furthermore, 
	\bes
		\phi_{(m-\nu+1)1}(W)=\det W^{[(m-\nu+1)1]}=\sum_{k=1}^{m-\nu+1}(-1)^{m-\nu+k+1}\wt w_{(m-\nu+1)k}\Delta_{(m-\nu+1)k},
	\ees
where $\Delta_{jk}$ is the $(j,k)$ minor of $W^{[(m-\nu+1)1]}$. Now, for $W=G(u)=df_\C(p)$, 
	\beas
		\wt w_{(m-\nu+1)k}&=&0, \, k=1,...,m-\nu+1,\\
				\Delta_{(m-\nu+1)k}&=&\begin{cases}
				i^\nu,& \text{if }k=1,\\
				i^{\nu-1},& \text{if }k=2,\\
				0,& \text{if }k=3,...,m-\nu+1.
				\end{cases}
	\eeas
Thus, 
	\bes
		\partl{\phi_{(m-\nu+1)1}}{ w_{jk}}(G(u))=\begin{cases}
		(-1)^{m-\nu}i^\nu,& \text{if }j=m-\nu+1, k=1,\\
		(-1)^{m-\nu+1}i^{\nu-1},& \text{if }j=m-\nu+1, k=2,\\
		0,& \text{for all other values of $j$ and $k$}.
	\end{cases}
	\ees
Repeating this computation for all $r\in\{m-\nu+1,...,n\}$ and $s\in\{1,...,\nu\}$, we obtain that 
	\bes 
E^{[rs]}=\left(\overline{\partl{\phi_{rs}}{ w_{jk}}}(G(u))\right)_{j=1,...,n}^{k=1,...,m}
		=\left(\begin{array}{c c c}		
		(-1)^{m}i^\nu\boldsymbol{\delta}_{rs} & (-1)^mi^{\nu-1}\boldsymbol{\delta}_{rs} 	
		& \mathbf{0}_{n\times(m-2\nu)}
 		\end{array}\right),
\ees
where $\boldsymbol{\delta}_{rs}$ is the $n\times \nu$ matrix whose $(r,s)$ entry is $1$ and all other entries are $0$. 

Next, we compute the space $DG_u(T_u(Y\times\{p\}))$. Given $A\in \text{Gl}(2n;\rl)$, we write 
	\bes
		A=\left(\begin{array}{c c c}
		\begin{array}{c c}
			a_{11} & b_{11}\\
			c_{11} & d_{11} 
		\end{array}
		& \cdots &
			\begin{array}{c c}
			a_{1n} & b_{1n}\\
			c_{1n} & d_{1n} 
		\end{array}\\
		\vdots & \cdots & \vdots \\
		\begin{array}{c c}
			a_{n1} & b_{n1}\\
			c_{n1} & d_{n1} 
		\end{array}
		& \cdots &
			\begin{array}{c c}
			a_{nn} & b_{nn}\\
			c_{nn} & d_{nn} 
		\end{array}\\
 \end{array}\right).
	\ees
Then, invoking \eqref{E:dfC}, we have that for $\mathbf A(z)=A(z)+a$,  $a\in\rl^{2n}$, 
\bes
d(\mathbf A\circ f)(p)=(A\circ df)(p)=\left(\begin{array}{c c c c c c| c   }
			a_{11} & \cdots & a_{1\nu} & b_{11} &\cdots & b_{1\nu}& \\
			c_{11} & \cdots & c_{1\nu} & d_{11} & \cdots & d_{1\nu} &  \\
			\vdots & \vdots & \vdots & \vdots
			& \vdots  & \vdots & \mathbf{0}_{2n\times (m-2\nu)}  \\
			a_{n1} & \cdots & a_{n\nu} & b_{n1} & \cdots & b_{n\nu} &  \\
			c_{n1} & \cdots & c_{n\nu} & d_{n1} & \cdots & d_{n\nu} & 
  \end{array}\right).\\
\ees
and
	\bes
		G({\mathbf A},p)=\left(A\circ df\right)_\C(p)=
\left(\begin{array}{c c c c c c| c   }
			a_{11}+i c_{11} & \cdots & a_{1\nu}+ic_{1\nu} 
			& b_{11}+id_{11} & \cdots & b_{1\nu}+id_{1\nu}& \\
			\vdots & \vdots & \vdots  
				&\vdots & \vdots & \vdots & \mathbf{0}_{n\times (m-2\nu)}  \\
			a_{n1}+ic_{n1} & \cdots & a_{n\nu}+ic_{n\nu} 
				& b_{n1}+id_{n1} & \cdots & b_{n\nu}+id_{n\nu} &  \\
  \end{array}\right).
	\ees
Now note that for $r\in\{n-\nu+1,...,n\}$ and $s\in\{1,...,\nu\}$, we have
	\bes
		(-1)^mi^{\nu-1}\left(\partl{G}{c_{rs}}(u)+\partl{G}{a_{rs}}(u)\right)=E^{[rs]},
	\ees
and thus $N_{G(u)}\operatorname{Hom}_\C^\nu(P_p,Q_q)\subset DG_u(T_uG(Y\times p))$, which establishes the transversality claim. The result now follows from the parametric transversality theorem, and the fact that 
$$
\text{codim}_{\C} (\operatorname{Mat}^\nu_{n\times m}(\C))=\nu(n-m+\nu).
$$
\end{proof}

\section{Some notions of degeneracy of order-one CR singularities}\label{S:nondeg}

In this section, we introduce some notions of (non)degeneracy for CR-singular points of order $1$, and describe a way to verify these notions in terms of local coordinates.

Assume that $\calS_1$ is a submanifold of $\M$. The following vector bundles over $\calS_1$ will play a role in the subsequent 
discussion:
\begin{itemize}
 \item $\mathcal H$, the real subbundle of $T\M$ over $\calS_1$, given by $\mathcal H_p=f^*(H_{f(p)}M)$,
\item  $N(\calS_1,\M)$, the normal bundle of $\calS_1$ in $\M$, given by $N_p(\calS_1, \M)=T_p\M/T_p \calS_1$, $p\in \calS_1$,
\item $K_\Phi$, the bundle over $\calS_1$ given by $K_p=\operatorname{Hom}_\C(\ker \Phi(p),{\coker}\Phi(p))$.
\end{itemize}
Note that $\mathcal H$ and $K_\Phi$ are complex bundles over $\calS_1$ of ranks $1$ and $n-m+1$, respectively. 
Finally, let $\pi: \mathcal H\rightarrow N(\calS_1,\M)$ denote the restriction of the projection map $T\M\rightarrow N(\calS_1,\M)$ to $\mathcal H$. 

\begin{defn}\label{D:Web} Let $\M$ be as above. Assume that $\calS_1$ is a submanifold of $\M$. We say that $p\in \calS_1$ is
\begin{itemize}
\item [($\mathcal W_0$)] a {\em totally degenerate} CR-singular point of $\M$ if $\dim_{\rl}\pi(\mathcal H_p)\cap N_{p}(\calS_1,\M)=0$; 
\item [($\mathcal W_1$)] a {\em parabolic} CR-singular point of $\M$ if $\dim_{\rl}\pi(\mathcal H_p)\cap N_{p}(\calS_1,\M)=1$;
\item [($\mathcal W_2$)] a {\em Webster nondegenerate} CR-singular point of $\M$ if $\dim_{\rl}\pi(\mathcal H_p)\cap N_{p}(\calS_1,\M)=2$.
\end{itemize}
We denote the set of totally degenerate, parabolic, and Webster nondegenerate points of $\M$ by 
$\mathcal W_0$, $\W_{1}$ and $\W_{2}$, respectively. 
\end{defn}

\begin{rem} Note that the submanifold $S_1=f(\calS_1)$ fails to be totally real precisely at the points in $f(\W_0)$. In the case of $m=n$, Webster nondegenerate points are referred to as elliptic or hyperbolic points in \cite{Web}, depending on whether $\pi$ is orientation-reversing or not. 
\end{rem}

A different notion of degeneracy of CR-singular points is discussed by Coffman in \cite{C}. However, this is done entirely in terms 
of local coordinates. We give a coordinate-free interpretation of Coffman nondegeneracy. We first make an observation about 
$N(\calS_1,\M)$ for a generic embedding~$f$. 

\begin{lemma}\label{L:N(S_1,M)} 
Suppose the section $\Phi_f: \M\rightarrow \operatorname{Hom}_\C(T^{\C\!}\M,\M\times\Cn)$ is transverse to $\operatorname{Hom}^1_\C(T^{\C\!} \M,\M\times\Cn)$. Then, $N(\calS_1, \M)\cong K_\Phi$ is a complex vector bundle over $\calS_1$ of rank $n-m+1$. 
\end{lemma}

\begin{proof} 
We recall two facts about fibre bundles over manifolds. Given two complex vector bundles $P$ and $Q$ over $\M$ of rank 
$m\leq n$, it is known that $\operatorname{Hom}^\nu_\C(P,Q)$ is a fiber subbundle of $\operatorname{Hom}_\C(P,Q)$ 
of complex codimension $\nu(n-m+\nu)$. Moreover, the fiber of the normal bundle to $\operatorname{Hom}_\C^1(P,Q)$ in 
$\operatorname{Hom}(P,Q)$ at $A$ is isomorphic to $\operatorname{Hom}_\C(\ker A,{\coker} A)$. See (\cite{GG}, Chapter VI \S 1). 

Next, if $g:X\rightarrow Y$ is a smooth map between differentiable manifolds, and $g$ is transverse to some submanifold $Z\subset Y$ with $\codim Z\leq \dim X$, then the normal bundle of $W=g^{-1}(Z)$ in $X$, denoted by $N(W,X)$, is the pullback of the normal bundle $N(Z,Y)$ of $Z$ in $Y$. 

Applying these facts to $g=\Phi$, $Y=\operatorname{Hom}_\C(T^{\C\!}\M,\M\times\Cn)$, and $Z=\operatorname{Hom}_\C^1(T^{\C\!}\M,\M\times\Cn)$ yields the result. 
\end{proof}

\begin{defn}\label{D:Coff} Suppose $\Phi_f: \M\rightarrow \operatorname{Hom}_\C(T^{\C\!} \M,\M\times\Cn)$ is transverse to 
$\operatorname{Hom}_\C^1(T^{\C\!} \M,\M\times\Cn)$ and $\mathcal S_1=\Phi_f^{-1}(\operatorname{Hom}_\C^1(T^{\C\!} \M,\M\times\Cn))$ is nonempty. Let $\pi:\mathcal H\rightarrow N(\calS_1, \M)$ be the projection map as above. We say that $p\in \calS_1$ is a {\em Coffman nondegenerate} CR-singular point of $\M$ if and only if $\pi(\mathcal H_p)$ is a totally real $2$-plane in the complex vector space 
$N_p(\calS_1, \M)$. Alternatively, we can complexify $\mathcal H$ to extend 
$\pi:\mathcal H\rightarrow N(\calS_1, \M)$ to $\pi_{\C\!}:\mathcal H\otimes\C\rightarrow N(\calS_1, \M)$, and so $p$ is a Coffman 
nondegenerate CR-singular point of $\M$ if and only if $\dim_{\C\!}\pi_\C(\mathcal H_p\otimes\C)=2$. More generally, we define
\bes
		{\mathcal C}_j=\{p\in \calS_1:\dim_{\C\!}\pi_\C(\mathcal H_p\otimes\C)=j\},\quad j=0,1,2.
\ees
\end{defn}

\begin{rem} Note that $\W_0= \calC_0$ and $\W_1\subseteq \calC_1$. 
\end{rem}

We now describe these notions of (non)degeneracy in terms of local coordinates, which is useful for computations. Assume that $2n\leq 3m-2$.
Let $f: \M \to \C^n$ be a smooth embedding such that $\calS_1$ is a smooth submanifold of $\M$. 
Following \cite{C}, for a given $p\in \calS=\calS_1$, we may assume that $f(p)=0$, and in a neighbourhood of $0\in M=f(\M)$ 
we have the following pre-normal form:
\bea\label{E:PN}
	M=\left\{z\in\C^n:\quad 
			\begin{aligned}
					y_s&=H_s(z_1,\overline z_1,x_2,...,x_{m-1}),\quad 2\leq s\leq m-1,\\
					z_u&=h_u(z_1,\overline z_1,x_2,...,x_{m-1}),\quad  m\leq u\leq n
			\end{aligned}
					 \right\},
\eea
where $H_s$, $h_u$, and their derivatives (with respect to $x_1,y_1,x_2,...,x_{m-1}$) vanish at $0$. We refer the reader to \cite[\S~3]{C} for preliminary normalizations that get rid of 
	\begin{itemize}
	\item [(i)] quadratic terms of type $z_1^2,z_1x_t,x_tx_r$ (and their conjugates) in the expansion of $H_s$,
	\item [(ii)]  quadratic terms of type $z_1^2,z_1x_t,x_tx_r$ in the expansion of $h_u$,
	\end{itemize}
	without affecting the other coefficients in the quadratic terms. Furthermore, by applying transformations of the form
	\be\label{E:cub}
		z_\alpha\mapsto z_\alpha+Az_1^3+\sum_{s=2}^{m-1}B^sz_1^2z_s+
			\sum_{s,t=2}^{m-1}C^{s,t}z_1z_sz_t+\sum_{s,t,r=2}^{m-1}D^{s,t,r}z_sz_tz_r,
	\ee
we can ensure that there are no 
	\begin{itemize}
	\item [(i)] cubic terms of type $z_1^3, z_1^2x_s, z_1x_sx_t$, and $x_sx_tx_r$ (and conjugates) in the expansion of $H_s$,
	\item [(ii)] cubic terms of type $z_1^3, z_1^2x_s, z_1x_sx_t$, and $x_sx_tx_r$ in the expansion of $h_u$,
	\end{itemize}
Thus, setting $\mathbf x=(x_1,y_1,x_2,...,x_{m-1})$ with $z_1=x_1+iy_2$, we may assume that  
	\bea\label{E:PN1}
		H_s(\mathbf x)&=&\beta_s|z_1|^2+\mu_sz_1^2\overline z_1+\overline\mu_s\overline z_1^2 z_1+
\sum_{t=2}^{m-1}\lambda_s^tz_1\overline z_1x_t
+O(4),\quad  2\leq s\leq m-1,\\
		h_u(\mathbf x)&=&\beta_u|z_1|^2+\gamma_u\overline z_1^2 
\label{E:PN2}
			+\sum_{s=2}^{m-1} \eps^{s}_{u}\overline z_1x_s
+\kappa_uz_1^2\overline z_1
		+\theta_uz_1\overline z_1^2+\pi_u\overline z_1^3+\notag \\
		&&+\sum_{s=2}^{m-1}\left(\varphi_u^sz_1\overline z_1x_s+\psi_u^s\overline z_1^2x_s+\sum_{t=2}^{m-1}\sigma_u^{s,t}\overline z_1x_sx_t\right)+O(4),\quad  m\leq u\leq n,
	\eea
for some constants $\beta_s,\beta_u,\lambda_s^t\in\rl$ and $\gamma_s,\mu_s,\gamma_u, \eps_u^s,\kappa_u,\theta_u,\pi_u,\varphi^s_u,\psi^s_u,\sigma_u^{s,t}\in\C$. In particular, we may take $\M=\rl^m$ with coordinates $\mathbf x=(x_1,y_1,x)$, $p=0$, and $f=(f_1,...,f_n): \M\rightarrow\Cn$ to be locally given by
	\be\label{E:emb}
		\mathbf x\mapsto \Big(z_1,x_2+iH_2(z_1,\overline z_1,x),...,
				x_{m-1}+iH_{m-1}(z_1,\overline z_1,x),h_m(z_1,\overline z_1,x),...,h_n(z_1,\overline z_1,x)\Big).
	\ee
Now, observe that $\calS_1$ consists of those points in $\mathcal M=\rl^m$ at which the $n\times m$ matrix 
	\be\label{E:array} 
\renewcommand*{\arraystretch}{2}
		df_\C=
		\begin{pmatrix}
		1 & i & 0 & \cdots &0\\
		i\partl{H_2}{x_1} & i\partl{H_2}{y_1} & 1+i\partl{H_2}{x_2} &
	\cdots  & i\partl{H_2}{x_{m-1}}\\  
		\vdots & \vdots & \vdots & \ddots &\vdots\\
i\partl{H_{m-1}}{x_1} & i\partl{H_{m-1}}{y_1} & i\partl{H_{m-1}}{x_2} &
	\cdots  & 1+i\partl{H_{m-1}}{x_{m-1}}\\ 
	\partl{h_m}{x_1} & \partl{h_m}{y_1} & \partl{h_m}{x_2}&\cdots &\partl{h_m}{x_{m-1}}\\
		\vdots &\vdots &\vdots &\ddots &\vdots \\
	\partl{h_n}{x_1} & \partl{h_n}{y_1} & \partl{h_n}{x_2}&\cdots &\partl{h_n}{x_{m-1}}
	\end{pmatrix}	
	\ee
has (complex) rank $m-1$. On applying the row reductions
	\bes
		R_j\mapsto 
		\begin{cases}
			R_j-i\partl{H_j}{x_1}R_1,& 2\leq j\leq m-1,\\
			i\left(R_j-\partl{h_j}{x_1}R_1\right),& m\leq j\leq n,
		\end{cases}
	\ees
we see that it suffices to consider points in $\rl^m$ for which the following $(n-1)\times (m-1)$ matrix has rank $m-2$: 
\bes
\begin{tiny}
\left(\begin{array}{c|c c c}
 & & & \\
 2\partl{H_2}{\overline z_1} & & \\
\vdots &   & \mathbf I_{(m-2)}+\left(i\partl{H_j}{x_k}\right)_{j,k=2}^{m-1}&   \\
2\partl{H_{m-1}}{\overline z_1}  & & & \\
 & & & \\
\hline
 & & & \\
2\partl{h_m}{\overline z_1} & & & \\
 \vdots &   & \left(i\partl{h_u}{x_k}\right)_{\substack{m\leq u\leq n\\ 2\leq k\leq m-1}} &    \\
2\partl{h_n}{\overline z_1} & & & \\
 & & & \\
\end{array}\right)
=\left(\begin{array}{c|c c c}
 & & & \\
 2\partl{H_2}{\overline z_1} & & \\
\vdots &   & \mathbf I_{(m-2)}+O(2)
 &   \\
2\partl{H_{m-1}}{\overline z_1}  & & & \\
 & & & \\
\hline
 & & & \\
2\partl{h_m}{\overline z_1} & & & \\
 \vdots &   & \left(i\partl{h_u}{x_k}\right)_{\substack{m\leq u\leq n\\ 2\leq k\leq m-1}}  &    \\
2\partl{h_n}{\overline z_1} & & & \\
 & & & \\
\end{array}\right)
\end{tiny}
=
\left(\begin{array}{c |c}
\mathbf G & \mathbf H\\
\hline
\mathbf J & \mathbf K
\end{array}\right).
\ees
Since $\mathbf H$ is invertible near $\mathbf x=0$, the following row reductions
\bes
\left(\begin{array}{c c}
\mathbf I_{(m-2)} & \mathbf 0_{(m-2)\times(n-m+1)}\\
& \\
 -\mathbf K & \mathbf I_{(n-m+1)}
\end{array}\right)
\cdot 
\left(\begin{array}{c c}
\mathbf H^{-1} & \mathbf 0_{(m-2)\times(n-m+1)}\\
& \\
 \mathbf 0_{(n-m+1)\times(m-2)} & \mathbf I_{(n-m+1)}
\end{array}\right)\cdot 
\left(\begin{array}{c c}
\mathbf G & \mathbf H\\
\mathbf J & \mathbf K
\end{array}\right)
\ees
yields that, near $0$, 
\bes
	\text{rank}df_\C=m-1 
\iff \text{rank}\left(\begin{array}{c c}
\mathbf H^{-1}\mathbf G & \mathbf I\\
\mathbf J-\mathbf K\mathbf H^{-1}\mathbf G & \mathbf 0
\end{array}\right)=m-2
\iff \mathbf J-\mathbf K\mathbf H^{-1}\mathbf G \equiv\mathbf 0_{(n-m+1)\times 1}.
\ees
We have proved the following result.
\begin{lemma}\label{L:S1} Let $f:\mathcal M\rightarrow\Cn$ be as given by \eqref{E:PN}, \eqref{E:PN1}, \eqref{E:PN2} and \eqref{E:emb}. Then, there is a neighborhood $U$ of $\mathbf x=0$ such that 
	\be\label{E:CRS}
		\calS_1\cap U=\left\{\mathbf x\in U : B_u(\mathbf x)=0,\ \  m\leq u\leq n\right\},
	\ee
where $B_m,...,B_n$ are the entries of $\mathbf J-\mathbf K\mathbf H^{-1}\mathbf G$. In particular,  
	\bea
		B_u(\mathbf x)
		=\beta_uz_1+2\gamma_u\overline z_1+\sum_{s=2}^{m-1} \eps^{s}_{u} x_s-i\sum_{s=2}^{m-1}\eps^{s}_u\beta_s|z_1|^2+\kappa_uz_1^2+2\theta_u|z_1|^2+3\pi_u\overline z_1^2+\notag \\
		+\sum_{s=2}^{m-1}\left(\varphi_u^sz_1x_s+2\psi_u^s\overline z_1x_s+\sum_{t=2}^{m-1}\sigma_u^{s,t}x_sx_t\right)+
	O(3),\quad m\leq u\leq n. \label{E:CRSb}
	\eea
\end{lemma}

We also describe a vector field (up to O(1) terms) that spans $\mathcal H_\C$ over $\mathcal S_1$.

\begin{lemma}\label{L:spanH} Let $f:\mathcal M\rightarrow\Cn$ be as given by \eqref{E:PN}, \eqref{E:PN1}, \eqref{E:PN2} and \eqref{E:emb}. There is a section $\mathcal V$ of $\mathcal H_\C|_{\mathcal S_1}$ of the form  
\beas
	\mathcal V=\partl{}{z_1}+O(1)
\eeas
so that $(\mathcal H_\C)_{\mathbf x}=\text{span}_{\C\!}\left\{\mathcal V(\mathbf x)\right\}$ for $\mathbf x\in \calS_1$.
\end{lemma}
\begin{proof} Since $\ker (df_\C)=\overline{\mathcal H_\C}$ for any any embedding $f:\mathcal M\hookrightarrow\Cn$, we solve for a vector field $\mathcal V=a_1\partl{}{x_1}+b_1\partl{}{y_1}+\sum\limits_{s=2}^{m-1}a_s\partl{}{x_s}$, for smooth $\C$-valued $a_1,b_1,a_s$, so that  
	\bes
		(df_\C)\overline{\mathcal V}\equiv 0\quad \text{on }\mathcal S_1,
	\ees
for $df_\C$ as computed in \eqref{E:array}. 
\end{proof}

Finally, we are in the position to describe the various nondegeneracies introduced above, in terms of the coefficients appearing in \eqref{E:PN1} and \eqref{E:PN2}. 

\begin{prop}\label{P:LocCoord} Let $f:\mathcal M\rightarrow\Cn$ be as given by \eqref{E:PN}, \eqref{E:PN1}, \eqref{E:PN2} and \eqref{E:emb}. Let $B_m,...,B_n$ be as in Lemma~\ref{L:S1}, and set
	\be\label{E:defS1}
		\mathscr S=\left(\Re B_m,\Im B_m,...,\Re B_n,\Im B_n\right),
	\ee 
so that $\mathcal S_1\cap U=\{\mathbf x\in U:\mathscr S(\mathbf x)=0\}$ for some neighborhood $U$ of $0$. 
\begin{enumerate}
\item [(a)] The map $\Phi_f: \M\rightarrow\operatorname{Hom}_{\C}(P,Q)$ is transverse to $\operatorname{Hom}_\C^1(P,Q)$ at $p=0$ if and only if \be\label{E:nondeg}
\left(dB_m\wedge d\overline B_m\wedge\cdots\wedge dB_n\wedge d\overline B_n\right)(0)\neq 0, 
\ee
or, equivalently, the matrix 
\be\label{E:nondeg2}
d\mathscr S(0)=\begin{pmatrix}
 \beta_m+2 \Re \gamma_m & 2\Im \gamma_m & \Re \eps_m^2 & \cdots  & \Re \eps_m^{m-1}
\\
2\Im \gamma_m & \beta_m - 2\Re\gamma_m & \Im \eps_{m}^2 &  \cdots &  \Im \eps_{m}^{m-1} 
\\
\vdots & \vdots &\vdots &\ddots & \vdots
\\ 
\beta_n+2\Re \gamma_n &  2\Im \gamma_n & \Re \eps_n^2 & \cdots  & \Re \eps_n^{m-1} 
\\
2\Im \gamma_n & \beta_n - 2\Re\gamma_n & \Im \eps_n^2 & \cdots  & \Im \eps_n^{m-1}
\end{pmatrix}.
\ee
has real rank $2(n-m+1)$. 
\item [(b)] Assume that \eqref{E:nondeg} holds. The point $p=0$ is in $W_j$, $j=0,1,2,$ if and only if \
	\be\label{E:Web}
		\operatorname{rank}_{\rl\!}\begin{pmatrix}	
		2\Im\gamma_m & \beta_m-2\Re\gamma_m\\
		\beta_m+2\Re\gamma_m & 2\Im\gamma_m\\
		\vdots &\vdots\\
		2\Im\gamma_n & \beta_n-2\Re\gamma_n\\
		\beta_n+2\Re\gamma_n & 2\Im\gamma_n
		\end{pmatrix}=j. 
	\ee
\item [(c)] Assume that \eqref{E:nondeg} holds. The point $p=0$ is in $C_j$, $j=0,1,2,$ if and only if 
	\be\label{E:Coff}
		\operatorname{rank}_{\C\!}\begin{pmatrix}	
		\beta_m & \gamma_m\\
		\vdots &\vdots\\
		\beta_n & \gamma_n
		\end{pmatrix}=j. 
	\ee 
\end{enumerate}
\end{prop}
\begin{proof} First, we prove (a). Since our computations are local, we identify $\operatorname{Hom}_{\C\!}(T^{\C\!} \M,\M\times\C)$ 
with $\M\times \operatorname{Mat}_{n\times m}(\C)$ and $\operatorname{Hom}^1_\C(T^{\C\!} \M,\M\times\Cn)$ with 
$\M\times \operatorname{Mat}^1_{n\times m}(\C)$, where $\operatorname{Mat}_{n\times m}(\C)$ denotes the space of 
$n\times m$ complex matrices, and $\operatorname{Mat}^1_{n\times m}(\C)$ denotes the space of $n\times m$ complex 
matrices of rank $m-1$. The transversality of $\Phi=\Phi_f$ to $\operatorname{Hom}_\C^1(P,Q)$ at $p$ is equivalent to 
	\bes
		(d\Phi)_p(T_p\M)+T_{df_\C(p)}\operatorname{Mat}^1_{n\times m}(\C)=T_{df_\C(p)}\operatorname{Mat}_{n\times m}(\C),
	\ees
as real subspaces, or that the projection of $(d\Phi)_p(T_p \M)$ onto $N_{df_\C(p)}\operatorname{Mat}^1_{n\times m}(\C)$ is surjective. 

As noted in the proof of Proposition~\ref{P:PertI}.a, 
\bes
df_\C(0) =
\left(\begin{array}{c | c}
		\begin{array}{c|c}
			1 & i\\
				\hline
			\mathbf{0}_{(m-2)\times 1} & \mathbf{0}_{(m-2)\times 1} 
		\end{array}
		& 
		\begin{array}{c}
			\mathbf{0}_{1\times(m-2)} \\
				\hline
			\mathbf{I}_{(m-2)\times (m-2)}
		\end{array}\\
		\hline
		  		\mathbf{0}_{(n-m+1)\times 2} & \mathbf{0}_{(n-m+1)\times(m-2)}
 \end{array}\right) ,
\ees
and 
\be\label{E:sigma}
N_{df_\C(0)}\operatorname{Mat}^1_{n\times m}(\C)=\text{span}_{\C\!}\left\{\sigma_r=\left(\begin{array}{c | c}		
			\mathbf{0}_{(m-1)\times 2} & \mathbf{0}_{(m-1)\times (m-2)}\\
		\hline
		  	\begin{array}{c|c}
		i\boldsymbol{\delta}_{r} & \boldsymbol{\delta}_{r} 
		\end{array} & \mathbf{0}_{(n-m+1)\times(m-2)}
 \end{array}\right):r=1,...,n-m+1\right\},
\ee
where $\boldsymbol{\delta}_r$ is the $(n-m+1)$ column matrix whose $r$-th entry is $1$, and the rest are $0$. 

Next, note that
	\be\label{E:push}
	(d\Phi)_0(T_0 \M)=\text{span}_\rl
\left\{\frac{\partial \Phi}{\partial x_1}(0), \frac{\partial \Phi}{\partial y_1}(0), \frac{\partial \Phi}{\partial x_2}(0),\dots, 
\frac{\partial \Phi}{\partial x_{m-1}}(0)\right\}.
\ee
The surjectivity of the projection of $(d\Phi)_0(T_0 \M)$ onto $N_{df_\C(0)}\operatorname{Mat}^1_{n\times m}(\C)$ is equivalent to the matrix
	\bes
		\Lambda=\begin{pmatrix}[1.5]
\Re\Big\langle \frac{\partial \Phi}{\partial x_1}(0),\overline\sigma_1\Big\rangle &
\Im\Big\langle \frac{\partial \Phi}{\partial x_1}(0),\overline\sigma_1\Big\rangle & 
\dots & 
\Re\Big\langle \frac{\partial \Phi}{\partial x_1}(0),\overline\sigma_{\ell}\Big\rangle &\Im\Big\langle \frac{\partial \Phi}{\partial x_1}(0),\overline\sigma_{\ell}\Big\rangle
  \\
\Re\Big\langle \frac{\partial \Phi}{\partial y_1}(0),\overline\sigma_1\Big\rangle &
\Im\Big\langle \frac{\partial \Phi}{\partial y_1}(0),\overline\sigma_1\Big\rangle & 
\dots & 
\Re\Big\langle \frac{\partial \Phi}{\partial y_1}(0),\overline\sigma_{\ell}\Big\rangle &\Im\Big\langle \frac{\partial \Phi}{\partial y_1}(0),\overline\sigma_{\ell}\Big\rangle
  \\
\Re\Big\langle \frac{\partial \Phi}{\partial x_2}(0),\overline\sigma_1\Big\rangle &
\Im\Big\langle \frac{\partial \Phi}{\partial x_2}(0),\overline\sigma_1\Big\rangle & 
\dots & 
\Re\Big\langle \frac{\partial \Phi}{\partial x_2}(0),\overline\sigma_{\ell}\Big\rangle &\Im\Big\langle \frac{\partial \Phi}{\partial x_2}(0),\overline\sigma_{\ell}\Big\rangle
  \\
\vdots &\vdots & \cdots & \vdots &\vdots\\
\Re\Big\langle \frac{\partial \Phi}{\partial x_{m-1}}(0),\overline\sigma_1\Big\rangle &
\Im\Big\langle \frac{\partial \Phi}{\partial x_{m-1}}(0),\overline\sigma_1\Big\rangle & 
\dots & 
\Re\Big\langle \frac{\partial \Phi}{\partial x_{m-1}}(0),\overline\sigma_{\ell}\Big\rangle &\Im\Big\langle 
\frac{\partial \Phi}{\partial x_{m-1}}(0),\overline\sigma_{\ell}\Big\rangle
  \\
\end{pmatrix}
	\ees
having (real) rank $2\ell$, where $\ell=n-m+1$. Here $\Big\langle \cdot,\cdot\Big\rangle$ is the standard complex bilinear pairing on $\operatorname{Mat}_{n\times m}(\C)$. From the description of $\sigma_r$ in \eqref{E:sigma}, we only need to compute the $(u,1)$
and $(u,2)$ entries of the matrices in \eqref{E:push}, where $u=m,\dots,n$. A direct computation using~\eqref{E:PN},~\eqref{E:PN1} and~\eqref{E:PN2} yields that, if $\Phi(\mathbf x)=df_\C(\mathbf x)=\Big(\Phi_{jk}(\mathbf x)\Big)_{j=1,...,n}^{k=1,...,m}$, then 
\bea
\Phi_{u1}(\mathbf x)& =& (2\beta_u +2 \Re \gamma_u)x_1 + 2 \Im \gamma_u y_1 + \sum_{s=2}^{m-1} (\Re \eps_u^s) x_s + \notag\\
&& \qquad i\left(2\Im \gamma_u x_1 - 2\Re \gamma_u y_1 + \sum_{s=2}^{m-1} (\Im \eps_u^s) x_s \right) + O(2) \label{E:phi1},\\
\Phi_{u2}(\mathbf x) &=& (2\beta_u -2\Re \gamma_u)y_1 +2 \Im\gamma_u x_1 + \sum_{s=2}^{m-1} (\Im \eps_u^s) x_s \notag \\
&& \qquad -i\left( 2\Im \gamma_u y_1 + 2\Re \gamma_u x_1+ \sum_{s=2}^{m-1}(\Re \eps_u^s )x_s\right)+ O(2)\label{E:phi2},
\eea
for $u=m,...,n$. Now, using \eqref{E:sigma}, \eqref{E:phi1} and \eqref{E:phi2} to compute $\Lambda$, we obtain that 
\bes
\Lambda=\begin{pmatrix}
4 \Im \gamma_m & -2\beta_m-4 \Re \gamma_m & 
\cdots &  4 \Im \gamma_n & -2\beta_n - 4 \Re \gamma_n \\
 2\beta_m - 4 \Re\gamma_m& -4 \Im \gamma_m & 
 \cdots &  2\beta_n - 4 \Re\gamma_n& -4 \Im \gamma_n \\
 2\Im \eps_{m}^2 &  -2\Re \eps_m^2 & 
 \cdots & 2\Im \eps_{n}^2  & -2\Re \eps_n^2 \\
 2\Im \eps_{m}^3 &  -2\Re \eps_m^3 & 
 \cdots & 2\Im \eps_{n}^3  & -2\Re \eps_n^3 \\
\vdots &\vdots &\ddots &\vdots &\vdots  \\
 2\Im \eps_{m}^{m-1} &  -2\Re \eps_m^{m-1} & 
 \cdots & 2\Im \eps_{n}^{m-1}  & -2\Re \eps_n^{m-1} \\
\end{pmatrix}.
\ees
Thus, the transversality of $\Phi_f$ to $\operatorname{Hom}_\C^1(P,Q)$ at $p=0$ is equivalent to the condition that the rank of the above matrix is $2(n-m+1)$, which in turn is equivalent to the condition~\eqref{E:nondeg2}.

Next, we establish \eqref{E:Web} and \eqref{E:Coff}. Let $\mathcal V$ be a section of $\mathcal H_\C|_{\mathcal S_1}$ such that $\mathcal H_\C =\operatorname{span}_{\C\!}\{\mathcal V\}$ on $S_1$. Assume that condition~\eqref{E:nondeg} holds. Then $dB_m\wedge d\overline{B_m}\wedge\cdots\wedge B_n\wedge d\overline{B_n}\neq 0$ in $U$ (shrinking further, if necessary). This implies that the $2(n-m+1)$ vectors  
	\bes	
	\nabla(\Im B_m)(\mathbf x),-\nabla(\Re B_m)(\mathbf x),
		...,\nabla(\Im B_n)(\mathbf x),-\nabla(\Re B_n)(\mathbf x),
	\ees
form a real basis of $N_{\mathbf x}(\mathcal S_1,\mathcal M)$ at any $\mathbf x\in\mathcal S_1\cap U$
in that order. Furthermore, 
$$
\mathcal H_{\mathbf x}=\text{span}_{\rl}\left\{\Re \mathcal V(\mathbf x),\Im \mathcal V(\mathbf x)\right\} .
$$
Thus, the projection map $\pi:\mathcal H_{\mathbf x}\rightarrow N_{\mathbf x}(\calS_1, \M)$ with respect to this choice of bases is given by $\Pi(\mathbf x)$, where
	\bes
			\Pi=\begin{pmatrix}	
		(\Re \mathcal V)(\Im B_m) & (\Im \mathcal V)(\Im B_m)\\
		(\Re \mathcal V)(-\Re B_m) & (\Im \mathcal V)(-\Re B_m)\\
		\vdots &\vdots\\
		(\Re \mathcal V)(\Im B_n) & (\Im \mathcal V)(\Im B_n)\\
		(\Re \mathcal V)(-\Re B_n) & (\Im \mathcal V)(-\Re B_n)\\
		\end{pmatrix},
	\ees
 and the complexification $\pi_\C: \mathcal H_{\mathbf x}\otimes\C\rightarrow N_{\mathbf x}(\calS_1, \M)$ is given by 
	\beas
	\Pi_\C&=&
		\begin{pmatrix}	
		(\Re \mathcal V)(\Im B_m-i\Re B_m) & (\Im \mathcal V)(\Im B_m-i\Re B_m)\\
		\vdots &\vdots\\
		(\Re \mathcal V)(\Im B_n-i\Re B_n) & (\Im \mathcal V)(\Im B_n-i\Re B_n)
		\end{pmatrix}
		=_{\operatorname{CR}}
				\begin{pmatrix}	
		\mathcal VB_m(\mathbf x) & \overline {\mathcal V} B_m(\mathbf x)\\
		\vdots &\vdots\\
		\mathcal VB_n(\mathbf x) & \overline {\mathcal V} B_n(\mathbf x)\\
		\end{pmatrix},
	\eeas
where ``$=_{\operatorname{CR}}$" denotes equality up to elementary column operations (over $\C$).

By Definitions~\ref{D:Web} and~\ref{D:Coff}, respectively, $p=0\in W_j$ if and only if $\operatorname{rank}_{\rl\!}\Pi(0)=j$, and $p=0\in C_j$ if and only if $\operatorname{rank}_{\C\!}\Pi_\C(0)=j$. Choosing $\mathcal V$ as granted by Lemma~\ref{L:spanH}, and using the formulas for $B_j$'s in \eqref{E:CRSb}, we obtain that the former condition is equivalent to condition \eqref{E:Web}, and the latter is equivalent to condition \eqref{E:Coff} (after elementary column and row reductions). 
\end{proof}
We now describe the defining functions (up to $O(2)$-terms) of the set $\mathcal C_1$ in $\mathcal M$. 

\begin{lemma}\label{L:C1} Let $f:\mathcal M\rightarrow\Cn$ be as given by \eqref{E:PN}, \eqref{E:PN1}, \eqref{E:PN2} and \eqref{E:emb}. Let $B_m,...,B_n$ as in Lemma~\ref{L:S1}. Suppose \eqref{E:nondeg} holds. Then, shrinking $U$ further, if necessary, we have
	\be\label{E:deC1}
		\mathcal C_1\cap U=
		\left\{\mathbf x\in \mathcal S_1\cap U:	\operatorname{rank}_{\C\!}			\begin{pmatrix}	
		\mathcal VB_m(\mathbf x) & \overline {\mathcal V} B_m(\mathbf x)\\
		\vdots &\vdots\\
		\mathcal VB_n(\mathbf x) & \overline {\mathcal V} B_n(\mathbf x)\\
		\end{pmatrix}=1\right\}.
	\ee
In particular, if $0\in\mathcal C_1$, we may assume that 
	\be\label{E:defC1}
		\mathcal C_1\cap U=
		\left\{\mathbf x\in U: \mathscr C(\mathbf x)=0\right\}, 
	\ee
where, $\mathscr C=\left(\Re B_m,\Im B_m,...,\Re B_n,\Im B_n,\Re \wt B_{m+1},\Im \wt B_{m+1},...,\Re \wt B_n,\Im\wt B_n\right)$, with 
\begin{itemize}
\item [$*$] $B_m,...,B_n$ as in \eqref{E:defS1}, $(\beta_m,\gamma_m)\neq (0,0)$, $(\beta_{m+1},\gamma_{m+1})=\cdots=(\beta_n,\gamma_n)=(0,0)$, and  
\item [$*$] $\wt B_j(\mathbf x)=\wt \beta_j z_1+2\wt \gamma_j \overline{z_1}+\sum\limits_{s=2}^{m-1}\wt \eps_j^sx_s+O(2)$, with  
	\beas
		\wt\beta_j
&=&2\beta_m\theta_j-i\beta_m\sum_{s=2}^{m-1}\eps_j^s\beta_s-4\gamma_m\kappa_j,\quad 
\wt\gamma_j=		3\beta_m\pi_j-2\gamma_m\theta_j+i\gamma_m\sum_{s=2}^{m-1}\eps_j^s\beta_s,\\ 
\wt\eps_j^s&=&2\left(\beta_m\psi_j^s-\gamma_m\varphi_j^s\right),\qquad m+1\leq j\leq n.
	\eeas
\end{itemize} 
\end{lemma}
\begin{proof} The proof of \eqref{E:deC1} is implicit in the proof of Proposition~\ref{P:LocCoord}. Now, assume that $0\in\mathcal C_1$. Let 
	\beas
		A(\mathbf x)&=&\begin{pmatrix}	
		\mathcal VB_m(\mathbf x) & \overline {\mathcal V} B_m(\mathbf x)\\
		\vdots &\vdots\\
		\mathcal VB_n(\mathbf x) & \overline {\mathcal V} B_n(\mathbf x)\\
		\end{pmatrix}
	\eeas
for the choice of $\mathcal V$ as granted by Lemma~\ref{L:spanH}. Then,
	\beas
		\mathcal VB_j(\mathbf x)
&=&\beta_j+(2\theta_j-i\sum_{s=2}^m\eps_j^s\beta_j)\overline z_1+2\kappa_jz_1+\sum_{s=2}^{m-1}\varphi_j^sx_s+O(2),\\
		\overline{\mathcal V}B_j(\mathbf x)
&=&2\gamma_j+(2\theta_j-i\sum_{s=2}^m\eps_j^s\beta_j)z_1+6\pi_j\overline z_1
		+2\sum_{s=2}^{m-1}\psi_j^sx_s+O(2).
	\eeas
Since $0\in\mathcal C_1$, we have that the rank of
	\bes
		A(0)=\begin{pmatrix}	
		\beta_m & 2\gamma_m\\
		\vdots &\vdots\\
		\beta_n & 2\gamma_n
		\end{pmatrix}\ \text{is }1. 
	\ees
Thus, one of the rows in $A$ must be non-zero. Suppose, the row corresponding to index $j$ is non-zero. Then, by composing $f$ with a biholomorphism of $\Cn$ that swaps $z_m$ with $z_j$, we have that $(\beta_m,2\gamma_m)\neq (0,0)$. Next, by the rank condition, we have the existence of $\lambda_j\in\C$ such that
	\bes	
		(\beta_j,2\gamma_j)=\lambda_j(\beta_m,2\gamma_m),\quad m+1\leq j\leq n. 
	\ees
Then, by composing $f$ with a biholomorphism of $\Cn$ that maps $z_j\mapsto z_j-\lambda_jz_m$, $m+1\leq j\leq n$, we can ensure that $(\beta_j,\gamma_j)= (0,0)$ for all $m+1\leq j\leq n$.

Finally, by continuity of the matrix-valued function $A$, there is a $\delta>0$ such that the first row of $A(\mathbf x)$ is non-zero for $\mathbf x\in\mathbb B(0,\delta)$. Thus, 
		\bes
		\mathcal C_1\cap \mathbb B(0,\delta)=
		\left\{\mathbf x\in \mathcal S_1\cap\mathbb B(0,\delta):	\mathcal V B_m(\mathbf x)\overline{\mathcal V} B_k(\mathbf x)-\overline{\mathcal V} B_m(\mathbf x)\mathcal V B_k(\mathbf x)=0, m+1\leq k\leq n\right\},
\ees
which yields \eqref{E:defC1}, since $(\beta_j,\gamma_j)= (0,0)$ for all $m+1\leq j\leq n$.

\end{proof}
 
We now compare Definition~\ref{D:Coff} with Coffman's notion of nondegenerate CR singularities. Suppose $f$ is an in \eqref{E:PN}. Following Coffman in \cite{C}, $M$ is said to satisfy the {\em first nondegeneracy condition} at $p$ if condition 
\eqref{E:Coff} holds with $j=2$. Under this assumption, after a holomorphic change of coordinates, $f$ can be assumed to satisfy the following conditions:
	\bea\label{E:redn}
		\beta_m=\cdots=\beta_{n-1}=\gamma_m=\cdots=\gamma_{n-2}=\gamma_n&=&0
			,\notag\\
		\beta_n=\gamma_{n-1}&=&1,\\
		\eps^2_{n-1}=\cdots=\eps^{m-1}_{n-1}&=&0.\notag
	\eea
Then, $M$ is said to satisfy the {\em second nondegeneracy condition} at $p$ if
	\be\label{E:Coff2}
		\text{rank}_\rl
		\begin{pmatrix}
		\Re\eps^2_m & \cdots & \Re\eps_m^{m-1}\\
		 \Im\eps^2_m & \cdots & \Im\eps_m^{m-1}\\
			\vdots & \ddots &\vdots\\
			\Re\eps^2_{n-2} & \cdots & \Re\eps_{n-2}^{m-1}\\
		 \Im\eps^2_{n-2} & \cdots & \Im\eps_{n-2}^{m-1}\\
\Re\eps^2_n & \cdots & \Re\eps_n^{m-1}\\
		 \Im\eps^2_n & \cdots & \Im\eps_n^{m-1}\\
		\end{pmatrix}_{2(n-m)\times (m-2)}=2(n-m).
	\ee
If both the nondegeneracy conditions hold at $p$, we say that $M$ satisfies {\em Coffman's nondegeneracy conditions} at $p$. We show that this is essentially equivalent to $p$ being a Coffman nondegenerate point in the sense of Definition~\ref{D:Coff}.

\begin{lemma}\label{L:NDcond} 
Let $f$ be as in \eqref{E:PN}. Then, $M$ satisfies Coffman's nondegeneracy conditions at $p$ if and only if 
condition~\eqref{E:nondeg2} holds, and condition \eqref{E:Coff} holds with $j=2$. 
\end{lemma}   
\begin{proof} It suffices to show that if condition \eqref{E:Coff} holds with $j=2$, then \eqref{E:Coff2} and \eqref{E:nondeg2} are equivalent conditions. In order to check \eqref{E:Coff2}, one first performs a local biholomorphic change of coordinates so that \eqref{E:redn} holds. On the other hand, condition \eqref{E:nondeg2} is invariant under local biholomorphic changes of coordinates at $p$ since it is equivalent to the fact that $\Phi_f: \M\rightarrow\operatorname{Hom}_\C(P,Q)$ is transverse to $\operatorname{Hom}^1_\C(P,Q)$ at $p$.

Thus, we may assume that, the coefficients in \eqref{E:PN1} and \eqref{E:PN2} satisfy \eqref{E:redn}. Then, referring to \eqref{E:nondeg2},
	\bes
	d\mathscr S(0)=
	\begin{pmatrix}
 0 & 0 & \Re \eps_m^2 & \cdots  & \Re \eps_m^{m-1}
\\
\vdots & \vdots &\vdots &\ddots & \vdots\\
0 & 0 & \Im \eps_{n-2}^2 &  \cdots &  \Im \eps_{n-2}^{m-1} 
\\
2 & 0 & 0 &\cdots & 0\\
0& 2 & 0 & \cdots & 0\\ 
0 & 0 & \Re \eps_n^2 & \cdots  & \Re \eps_n^{m-1} 
\\
0& 0 & \Im \eps_n^2 & \cdots  & \Im \eps_n^{m-1}
\end{pmatrix}.
	\ees
The rank of the above $m\times 2(n-m+1)$ matrix being $2(n-m+1)$ is clearly equivalent to \eqref{E:Coff2}. 
\end{proof}

Thus, if $f: \M\rightarrow\Cn$ is an embedding granted by Proposition~\ref{P:PertI}, then any $p\in C_2$ satisfies Coffman's nondegeneracy conditions. Combining this with \cite[Proposition~3.3]{C}, we obtain the following result. 

\begin{prop}\label{P:Coff} Let $\frac{2}{3}(n+1)\leq m<n$. Let $f:\M\rightarrow\Cn$ be an embedding, in general position (as granted by Proposition~\ref{P:PertI}). Let $p\in\mathcal C_2$. Assume that $f$ is real-analytic in a neighbourhood of $p$. Then, there exists a 
holomorphic change of coordinates in a neighbourhood of $f(p)$, mapping $f(p)$ to $0$, and such that, in the new coordinates, $M$ is given near $0$ by 
	\beas
		\left\{(z_1,...,z_n)\in\C^{n}: \ \ 
		\begin{aligned}
	&  y_s=0,\ 2\leq s\leq m-1,\\
	& z_{m}=\overline z_1^2,\\
	&z_{m+1}=|z_1|^2+\overline z_1(x_2+ix_3),\\
	&z_\ell
		=\overline z_1 (x_{2(\ell-m)}+ix_{2(\ell-m)+1}),\ m+2\leq \ell\leq n\\
\end{aligned}
\right\}.
	\eeas
\end{prop}

\section{Some stratifications of the set of order-one CR singularities}\label{S:strat}
We now describe the generic global structure of the various sets of degeneracy introduced in the previous section. For this, we express $\mathcal W_j$'s and $\mathcal C_j$'s as level sets of some global maps defined on $\calS_1$. We first produce an appropriate complex bundle over 
$\operatorname{Hom}_\C^1(T^{\C\!} \M,\M\times\Cn)$. Throughout this discussion, let
	\bes
		P=T^{\C\!} \M\quad \text{and}\quad Q=\M \times \Cn.
	\ees
Given $p\in \M$ and $\varphi\in \operatorname{Hom}_\C^1(P_p,Q_p)$, let $H_\varphi=\ker \varphi$ and $N_\varphi=T_{\varphi}\operatorname{Hom}_\C(P,Q)/T_{\varphi}\operatorname{Hom}_\C^1(P,Q)\cong \operatorname{Hom}_\C(\ker \varphi,{\coker}\varphi)$. Then, there exists a real vector bundle 
\be\label{E:EPQ}
E(P,Q)\xrightarrow{\pi_E} \operatorname{Hom}_\C^1(P,Q)
\ee
 of rank $(4n-4m+4)$ given by $E_\varphi(P,Q)=\operatorname{Hom}_\rl(H_\varphi,N_\varphi)$. Since
 \be\label{E:comp} 
\operatorname{Hom}_\rl(H_\varphi,N_\varphi)\cong \operatorname{Hom}_\C(H_\varphi\otimes_\R \C,N_\varphi),\ee
$E(P,Q)$ is also a complex bundle of (complex) rank $(2n-2m+2)$. Given a real homomorphism $A:H_\varphi\rightarrow N_\varphi$, we denote by $A_\C$ its complexification (in $ \operatorname{Hom}_\C(H_\varphi\otimes\C,N_\varphi)$). We obtain two distinct stratifications: 
	\bes
		E(P,Q)=\sqcup_{j=0}^2 E_j(P,Q)=\sqcup_{j=0}^2F_j(P,Q)
	\ees
given by 
	\beas
		E_j(P,Q)|_\varphi		
&=&\{A\in E(P,Q)|_\varphi:\operatorname{rank}_{\rl\!} A=j\},\\
		F_j(P,Q)|_\varphi
&=&\{A\in E(P,Q)|_\varphi:\operatorname{rank}_{\C\!}A_{\C\!}=j\}, \quad j=0,1,2.
	\eeas
Note that $E_0(P,Q)=F_0(P,Q)$, $E_1(P,Q)\subsetneq F_1(P,Q)$, and $F_2(P,Q)\subsetneq E_2(P,Q)$ are real submanifolds such that $\overline E_k=\sqcup_{j=0}^kE_j$ and $\overline F_k=\sqcup_{j=0}^kF_j$. Furthermore,
\bea\label{E:codim}
	\codim_\rl E_0(P,Q)=\codim_\rl F_0(P,Q)&=&4n-4m+4 \notag\\
	\codim_\rl E_1(P,Q)&=&2n-2m+1\\
	\codim_\rl F_1(P,Q)&=&2n-2m.\notag	
\eea
Now assume that $\calS_1$ is a smooth submanifold of $\M$. We have that $\Phi=\Phi_f$ restricted to $\calS_1$ is a section of the 
fibre bundle $\operatorname{Hom}_\C^1(P,Q)|_{\calS_1}$. Complexifying $T\M$ extends 
$d\Phi:T\M\rightarrow T\operatorname{Hom}_\C(P,Q)$ to 
	\bes 
		(d\Phi)_\C:T^{\C\!} \M\rightarrow T\operatorname{Hom}_\C(P,Q).
	\ees
We now associate a second-order Gauss map $\Psi=\Psi_f:S_1\rightarrow E(P,Q)$ to $f$ as follows. For each $p\in S_1$, consider the composition
	\bes
		\psi_p:\ker \Phi(p)=(\overline {H_\C})_p\hookrightarrow 
		T^{\C\!}_p \M\xrightarrow{(d\Phi)_\C(p)} T_{\Phi(p)}\operatorname{Hom}_\C(P,Q)\xrightarrow{\pi_p} N_{\Phi(p)}.
	\ees   
Then $\Psi_f:p\rightarrow\psi_p$ gives a map $\Psi_f:S_1\rightarrow E(P,Q)$ 
such that the following diagram commutes:
\[ \begin{tikzcd}
\M\supset \calS_1 \arrow{rd}{\Phi_f} \arrow{r}{\Psi_f}  & E(P,Q)\arrow{d}{\pi_E} \\%
& \operatorname{Hom}_\C^1(P,Q) &\hspace{-32pt}\subset \operatorname{Hom}(P,Q).
\end{tikzcd}
\]

\begin{prop}\label{P:PertII} 
Let $f: \M\hookrightarrow\Cn$ be a smooth embedding granted by Proposition~\ref{P:PertI}, i.e., 
$$
\Phi_f: \M\rightarrow \operatorname{Hom}_\C(P,Q)
$$ 
is transverse to $\operatorname{Hom}_\C^\nu(P,Q)$ for all $\nu=1,...,\lfloor{m/2}\rfloor$. 
Assume that $S_\nu=\emptyset$ for all $\nu\geq 2$. Thus, $\calS=\calS_1$ is a compact $(3m-2n-2)$-submanifold of $\M$. 
After a generic small perturbation,  $\Psi_f$ is transverse to $E_j(P,Q)$ and $F_j(P,Q)$ for $j=0,1$, i.e., 
	\begin{itemize}
		\item [(a)] $\W_0=\calC_0$ is a compact $(7m-6n-6)$-submanifold of $\calS=\calS_1$,
		\item [(b)] $\W_1$ is a (not necessarily closed) $(5m-4n-3)$-submanifold of $\calS$ with $\overline {\W_1}=\W_0\cup \W_1$, and
		\item [(c)]  $\calC_1$ is a (not necessarily closed) $(5m-4n-2)$-submanifold of $S$ with $\calC_0\cup \calC_1=\overline{\calC_1}$.
	\end{itemize}
In particular, if $6n>7m-6$, there exist embeddings $f: \M\rightarrow\Cn$ such that $\calS=\calS_1$ is a compact 
$(3m-2n-2)$-submanifold of $\M$, $\calC_0=\W_0$ is empty, $\calC_1$ is a compact $(5m-4n-2)$-submanifold of $\calS$, 
$\W_1$ is a compact $(5m-4n-3)$-submanifold of $\calC_1$. Furthermore,
\begin{itemize}
\item [(i)] any $\calC^2$-small perturbation of $f$ compactly supported on $\mathcal M\setminus\overline{\mathcal C_1}$ continues to have these 
properties, and 
\item [(ii)] any $\calC^3$-small perturbation of $f$ continues to have these 
properties. 
\end{itemize}
\end{prop}

\begin{proof}	As in the proof of Proposition~\ref{P:PertI}, we use the parametric transversality theorem. Let $f$ be as granted by Proposition~\ref{P:PertII}. Assume that $\calS_\nu=\emptyset $ for all $\nu\geq 2$. We denote $\calS$ by $\calS^f$ and 
$\calS_\nu$ by $\calS_\nu^f$ in this proof. 

Let $U(n)$ and $\mathbf P_2(\Cn)$ denote the space of $n\times n$ unitary matrices and the space of homogeneous degree 
$2$ polynomial maps on $\Cn$, respectively. Consider the following family of transformations:
\bes
	\wt Y=\{\mathbf A(z)=a+A' z+A''(z):a\in\Cn, A'\in \text{U}(n), A''\in\mathbf P_2(\Cn)\}.
\ees
For any $\mathbf A\in \wt Y$, if $A''$ is sufficiently small, then 
\begin{itemize}
\item [(i)] $\calS_\nu^{\mathbf A\circ f}=\emptyset$ for all $\nu\geq 2$, i.e., 
	$\calS^{\mathbf A\circ f}=\calS_1^{\mathbf A\circ f}$,
\item [(ii)] $\calS^{\mathbf A\circ f}$ is a compact $(3m-2n-2)$-submanifold of $\M$,
\item [(iii)] viewing any small tubular neighbourhood $\mathcal T$ of $\calS^f$ in $\M$ as a disk bundle over 
	$\calS^f$, there is a smooth regular section $\eta_{\mathbf A}: \calS^f\rightarrow \mathcal T$ whose image is precisely 
	$\calS^{\mathbf A\circ f}$.  
\end{itemize}
Let $Y\subset\wt Y$ denote the set of those $\mathbf A\in \wt Y$ whose corresponding $A''$ is sufficiently small 
so that conditions (i)-(iii) hold. Let $E(P,Q)$ be the bundle over $\operatorname{Hom}_\C^1(P,Q)$ defined in \eqref{E:EPQ}.

Now, consider the smooth map $G:Y\times \calS^f\rightarrow E(P,Q)$ given by 
	\bes
		G(\mathbf A, p)=\left(\Psi_{\mathcal A\circ f}\circ \eta_{\mathbf A}\right)(p).
	\ees
We show that $G$ is transverse to $Z\subset E(P,Q)$, where $Z$ is any one of the submanifolds $E_0(P,Q)=F_0(P,Q)$, $E_1(P,Q)$ and $F_1(P,Q)$. The result then follows from the parametric transversality theorem and the codimension count given in \eqref{E:codim}.

We first make some simplifications. By relabelling $\mathbf A\circ f$ as $f$, verifying the transversality of $G$ to $Z$ at $G(\mathbf A,p)\in Z$ is equivalent to verifying the transversality of $G$ to $Z$ at $G(I,p)\in Z$, where $I$ is the identity map on $\Cn$. Moreover, for any translation map $T$ and unitary map $U$ on $\Cn$, $G(U\circ T\circ \mathbf A,p)$ is transverse to $Z$ if and only if $G(\mathbf A,p)$ is transverse to $Z$. So, we assume that $f(p)=0$ and $M$ is given near $0$ by \eqref{E:PN}. Then, 
\be\label{E:difff}
df_\C(p) =
\left(\begin{array}{c | c}
		\begin{array}{c|c}
			1 & i\\
				\hline
			\mathbf{0}_{(m-2)\times 1} & \mathbf{0}_{(m-2)\times 1} 
		\end{array}
		& 
		\begin{array}{c}
			\mathbf{0}_{1\times(m-2)} \\
				\hline
			\mathbf{I}_{(m-2)\times (m-2)}
		\end{array}\\
		\hline
		  		\mathbf{0}_{(n-m+1)\times 2} & \mathbf{0}_{(n-m+1)\times(m-2)}
 \end{array}\right) .
\ee
Thus, $G(I,p)$ lies in 
	\bes 
		E(P,Q)|_{df_\C(p)}=\text{the fiber of $E(P,Q)$ over $df_\C(p)\in\operatorname{Hom}_\C^1(P,Q)$}.
	\ees

Next, we consider $G$ restricted to the following submanifold $Y'\subset Y$:
	\beas
	Y'=\left\{\mathbf A\in Y:\mathbf A(z_1,...,z_n)= \left(z_1,...,z_{m-1},z_m+b_m|z_1|^2+c_m\overline z_1^2,...,z_n+b_n|z_1|^2+c_n\overline z_1^2\right),\right.\\	
\left.	\text{for some }b_m,...,b_n,c_m,...,c_n\in\C\right\},
	\eeas
where we give $Y'$ the coordinates $(b,c)=(b_m,...,b_n,c_m,...,c_n)\in\C^{n-m+1}$. With this choice, $\eta_{\mathbf A}(p)=p$ and $G(\mathbf A,p)\in E(P,Q)|_{df_\C(p)}$ for all $\mathbf A\in Y'$. Since we work in this fixed fiber, it can be identified with $\operatorname{Mat}_{(2n-2m+2)\times 2}(\rl)$. From the computations in the proof of Proposition~\ref{P:LocCoord}.b, we obtain that 
	\be\label{E:G(A,p)}
	\Psi_{\mathbf A\circ f}=
		G(\mathbf A,p)=\begin{pmatrix}
		2\Im\gamma_m +\Im b_m+2\Im c_m & \beta_m+2\Re\gamma_m+\Re b_m-2\Re c_m\\
		\beta_m-2\Re\gamma_m+\Re b_m+2\Re c_m & 2\Im\gamma_m-\Im b_m+2\Im c_m\\
		\vdots &\vdots\\
		2\Im\gamma_n+\Im b_n+2\Im c_n & \beta_n+2\Re\gamma_n+\Re b_n-2\Re c_n\\
		\beta_n-2\Re\gamma_n+\Re b_n+2\Re c_n & 2\Im\gamma_n-\Im b_n+2\Im c_n
		\end{pmatrix},
	\ee
when $\mathbf A=(b,c)\in Y'$. It follows that 
\beas
	DG_{(I,p)}\left(T_{(I,p)}(Y'\times\{p\})\right)
&=&\text{span}_\rl \left\{\partl{G}{\Re b_j},\partl{G}{\Im b_j},\partl{G}{\Re c_j},\partl{G}{\Im c_j}:j=m,...,n\right\}\Big|_{(I,p)}\\
&=&
\operatorname{Mat}_{(2n-2m+2)\times 2}(\rl)\\
&=& T_{G(I,p)}E(P,Q).
\eeas
Thus, $G$ is transverse to $Z$ at $(I,p)$ for each of the relevant choices of $Z$. 

The last statement of the proposition follows from Hirsch~\cite[Lemma 2.3, p. 76]{H}.
\end{proof}
Analoguous to Proposition~\ref{P:LocCoord}.a, one obtains the following characterization of transversality of $\Psi_f$ to $F_1(P,Q)$. We skip the proof as it follows the same ideas as those of Proposition~\ref{P:LocCoord}.
 
\begin{prop}\label{P:transC1} Assume that $4n\leq 5m-2$. Let $f:\mathcal M\rightarrow\Cn$ be as given by \eqref{E:PN}, \eqref{E:PN1}, \eqref{E:PN2}, and \eqref{E:emb}. Suppose $\Phi_f: \M\rightarrow\operatorname{Hom}_{\C\!}(P,Q)$ is transverse to $\operatorname{Hom}_\C^1(P,Q)$ at $p=0$, and $0\in\mathcal C_1$. Assume $\mathcal C_1$ is as granted by Lemma~\ref{L:C1}, i.e.,  $\mathcal C_1=\{\mathbf x\in\M: \mathscr C(\mathbf x)=0\}$, where $\mathscr C$ is as in \eqref{E:defC1}. Then, the Gauss map $\Psi_f$ is transverse to $F_1(P,Q)$ if and only if 
the $(4n-4m+2)\times m$ matrix  
\be\label{E:TC1}
\renewcommand*{\arraystretch}{1.5}
d\mathscr C(0)=
	\begin{pmatrix}
	\beta_m+\Re\gamma_m & \Im\gamma_m & \Re\eps_m^2 & \cdots & \Re\eps_m^{m-1}\\
	\Im\gamma_m & \beta_m-\Re\gamma_m & \Im\eps_m^2 & \cdots & \Im\eps_m^{m-1}\\
	0 & 0 & \Re\eps_{m+1}^2 & \cdots & \Re\eps_{m+1}^{m-1}	\\
	0 & 0 & \Im\eps_{m+1}^2 & \cdots & \Im\eps_{m+1}^{m-1}	\\
	\vdots & \vdots & \vdots &\ddots &\vdots\\
		\Re\wt\beta_{m+1}+\Re\wt\gamma_{m+1} & \Im\wt\gamma_m-\Im\wt\beta_{m+1} & \Re\wt\eps_{m+1}^2 & \cdots & \Re\wt\eps_{m+1}^{m-1}\\
	\Im\wt\gamma_{m+1}+\Im\wt\gamma_{m+1} & \Re\wt\beta_{m+1}-\Re\wt\gamma_{m+1} & \Im\wt\eps_{m+1}^2 & \cdots & \Im\wt\eps_{m+1}^{m-1}\\
\vdots & \vdots & \vdots &\ddots &\vdots
	\end{pmatrix}
\ee
has full rank $4n-4m+2$. 
\end{prop}

\section{A perturbation result for totally real points}\label{S:pert} In this section, we give a refinement of a perturbation 
result due to Arosio and Wold~\cite{ArWo19}. Our result essentially yields that, if the set of CR singularities of a 
compact $m$-submanifold $M\subset\Cn$ ($m<n$) is already polynomially convex, then the manifold can be perturbed 
so that it is itself polynomially convex.

We first recall two lemmas from \cite{ArWo19}.

\begin{lemma}[{\cite[Lemma~3.1]{ArWo19}}]\label{L:1} 
Let $\M$ be a compact $\calC^\infty$ manifold (possibly with boundary) of dimension $m<n$ and let 
$f: \M\hookrightarrow\Cn$ be a totally real $\calC^\infty$ embedding. Let $K\subset\Cn$ be a 
polynomially convex compact  set, and $U$ be a neighborhood of $K$. Then, for any $k\geq 1$ and $\delta>0$, 
there exists a totally real $\calC^\infty$-smooth embedding $f_\delta: \M\rightarrow\Cn$ such that 
\begin{enumerate}
    \item  $||f_\delta-f||_{\calC^k(\M)}<\delta$,
    \item $f_\delta=f$ in a neighborhood of $f^{-1}(K)$,
    \item $\widehat{K\cup f_\delta(\M)}\subset U\cup f_\delta(\M)$.
\end{enumerate}
\end{lemma}

\begin{lemma}[\cite{ArWo19}]\label{L:2}
Let $\M$, $f$, $K$ and $U$ be as in Lemma~\ref{L:1}, and suppose that 
$\widehat{K\cup f(\M)}\subset U\cup f(\M)$. 
Then, there exists a constant $c=c(\M,f,K,U)>0$ such that for any smooth $\tilde f : \M \to \C^n$
\begin{equation}\label{e.cnb}
||f-\tilde f||_{\calC^1(\M)}<c \ \Longrightarrow \ \widehat{K\cup \tilde f (\M)}\subset U\cup \tilde f (\M).
\end{equation}
\end{lemma}

\begin{proof}
This is essentially a restatement of Lemma~3.3 in~\cite{ArWo19}, which
under the assumptions of the lemma gives $c>0$ such that 
$||f-\tilde f||_{\calC^1(\M)}<c$ implies $h({K\cup \tilde f(\M)})\subset U$. Here 
$h(X) = \overline{\widehat X \setminus X}$ is the essential hull. From this Lemma~\ref{L:2} follows.
\end{proof}


In this section we prove the following perturbation result that will be used in the proof of our main theorem. 

\begin{prop}\label{l.pertlemma}
Let $\M$ be an $m$-dimensional compact $\calC^\infty$ manifold (possibly with boundary). Let $n>m$ and 
$f: \M\hookrightarrow\Cn$ be a $\calC^\infty$-smooth embedding. Suppose $C\subset f(\M)$ is a polynomially convex compact set,
and $f(\M)\setminus C$ is a totally real manifold.
Then, for any $k\geq 1$ and any $\eps>0$, there is a $\calC^\infty$-smooth embedding $f_\eps: \M\rightarrow\Cn$, totally real on 
$M \setminus f^{-1}(C)$, such that 
    \begin{itemize}
        \item [$(a)$] $||f_\eps-f||_{\calC^k(\M)}<\eps$,
        \item [$(b)$] $f_\eps=f$ on $f^{-1}(C)$, and
        \item [$(c)$] $f_\eps(\M)$ is polynomially convex in $\Cn$.
    \end{itemize}
\end{prop}

\begin{proof} 
We set $M = f(\M)$.
Since $C$ is polynomially convex, there exists (see~\cite{St}) a smooth nonnegative plurisubharmonic function
$\phi(z)$ on $\C^n$ such that $\{\phi^{-1}(0)\}=C$ and $\phi$ is strictly plurisubharmonic on $\C^n \setminus C$. 
It follows that for any $b>0$, the sublevel set $\{\phi \le b\}$  is a compact polynomially convex neighbourhood of $C$
(we may restrict consideration to a connected component of $\{\phi \le b\}$ containing $C$ if necessary).
By applying Sard's theorem to the restriction of $\phi$ to $M$, for almost all $b>0$, the set $\{\phi=b \}\cap M$ is a 
closed smooth submanifold of $M$ of dimension $m-1$. Thus, by choosing an appropriate sequence $b_j \searrow  0$, 
we may construction 
an open neighbourhood basis $\{U_j=\{\phi < b_j\}\}_{j\in\N}$ in $\Cn$ such that for each $j\in\N$ we have
\begin{itemize}
    \item[($\imath$)] $U_{j+1}\Subset U_{j}$,
        \item[($\imath\imath$)] $\overline {U_j}$ is polynomially convex in $\Cn$,
    \item[($\imath\imath\imath$)] $M\setminus U_j$ is a smooth totally real manifold with boundary.
\end{itemize}
Let $\eps>0$ and $k \ge 1$ be given. The idea of the proof is the following: using Lemma~\ref{L:1}, for each $j \in \mathbb N$ 
construct a ${\calC^{k+j}(\M)}$-small perturbation $f_j$ of $f$ such that the polynomially convex hull of $M_j=f_j(\M)$ is contained in 
$U_j \cup M_j$. Further, in the inductive procedure, the map $f_{j+1}$ is chosen from a neighbourhood of $f_j$ in the space of
embeddings that satisfy~\eqref{e.cnb} with $f=f_j$ and $\tilde f = f_{j+1}$. The required map $f_\eps$ is then obtained as a pointwise
limit of $\{f_j\}$. 

Set $\M_j=\M\setminus f^{-1}\left(U_{j+1}\right)$, $K_j =\overline{U_j}$.
By Lemma~\ref{L:1} applied to $\M=\M_1$, $f=f|_{\M_1}$, $K=K_1$ and choosing some $U_0$ that satisfies the assumption on $U$, 
there exists a totally real $\calC^\infty$ embedding $f_1: \M_1\rightarrow\Cn$ such that 
\begin{enumerate}
    \item [$(i_1)$] $f_1=f$ on $f^{-1}(K_1)$,
    \smallskip
    \item [$(ii_1)$] $||f_1-f||_{\calC^k(\M_1)} = ||f_1-f||_{\calC^k(\M)}<\eps/4$,
    \smallskip
    \item [$(iii_1)$] $\widehat{K_1\cup f_1(\M_1)}\subset U_0\cup f_1(\M_1)$.
\end{enumerate}
Now, by Lemma~\ref{L:2} there exists a constant $c_1 = c (\M_1, f_1,K_1,U_0)$ such that for any smooth $\tilde f$ we have
\begin{equation}\label{e.cnb1}
||f_1-\tilde f||_{\calC^1(\M)}<c_1 \ \ \Longrightarrow \ \ 
\widehat{K_1\cup \tilde f(\M_1)} \subset U_0 \cup \tilde f(\M_1).
\end{equation}
Chose $\eps_2 = \min\{\eps_1/2, c_1/2\}$, and after shrinking $\eps_2$ further we may also assume that 
  $$
  K_2+\mathbb B(0,\eps_2)\subset U_{1}.
 $$
 Again, by Lemma~\ref{L:1} there exists $f_2: \M \to \C^n$ such that 
 \begin{enumerate}
    \item [$(i_2)$]  $f_1=f_2$ on $f^{-1}(K_2)$,
    \smallskip
    \item [$(ii_2)$] $||f_2-f_1||_{\calC^{k+1}(\M_2)} = ||f_2-f_1||_{\calC^{k+1}(\M)} <\eps_2$,
    \smallskip
    \item [$(iii_2)$] $\widehat{K_2\cup f_2(\M_2)}\subset U_1\cup f_2(\M_2)$.
\end{enumerate}
Note that by construction, the map $f_2$ satisfies~\eqref{e.cnb1}, and therefore, 
$ \widehat{K_1\cup f_2(\M_1)} \subset U_0 \cup f_2(\M_1)$.

Suppose now that there exist $\eps_r>0$ and maps $f_r: \M \to \C^n$, $r = 1, 2, \dots, j-1$, such that  for $r>1$ we have
\begin{itemize}
    \item [$(0_r)$] $\eps_r <\min\{\eps_{r-1}/2, c_{r-1}/2\}$ and   ${K_r}+\mathbb B(0,\eps_r)\subset U_{r-1}$. 
    \smallskip
    \item [$(ii_r)$]  $f_r=f_{r-1}$ on $f^{-1}\left({K_r}\right)$,
    \smallskip
    \item [$(ii_r)$] $||f_r-f_{r-1}||_{\calC^{k+r}(\M_r)}<\eps_r$,
    \smallskip
    \item [$(iii_r)$] $\widehat{{K_r}\cup f_r(\M_r)}\subset U_{r-1}\cup f_r(\M_r)$.
\end{itemize}
Then by Lemma~\ref{L:1}, there exist $\eps_{j}$ and an embedding $f_{j}: \M \to \mathbb C^n$ such that
\begin{enumerate}
    \item [$(0_{j})$] $\eps_{j} <\min\{\eps_{j-1}/2, c_{j-1}/2\}$ and   ${K_{j}}+\mathbb B(0,\eps_j)\subset U_{j-1}$. 
    \smallskip
    \item [$(i_j)$] $f_{j-1}=f_{j}$ on ${f}^{-1}\left({K_{j}}\right)$,
    \smallskip
    \item  [$(ii_j)$] $||f_{j-1}-f_j||_{\calC^{k+j}(\M_{j})} = ||f_{j-1}-f_j||_{\calC^{k+j}(\M)} <\eps_{j}$, 
    \smallskip
    \item [$(iii_j)$] $\widehat{{K_{j}}\cup f_{j}(\M_{j})}\subset U_{j-1}\cup f_{j}(\M_{j})$.
\end{enumerate}
Note that for any $1<r<j$ we have $||f _r - f_j ||_{\calC^1(\M)} < c_r$, and therefore, by Lemma~\ref{L:2}, 
\begin{equation}
\widehat{K_r\cup f_j(\M)}\subset U_{r-1} \cup f_j (\M) .
\end{equation}
 
 Consider now the limit of the sequence $\{f_j\}$. We claim that (1) the limit exists and is a smooth function, call it $f_\eps$, 
 and  (2) $f_\eps$ satisfies conditions (a)-(c) of the lemma.
 
 (1) The sequence $\{f_j\}$ is Cauchy in $\calC^{\tilde\eta}(\M)$ for any $\tilde\eta\ge 1$. To prove that assume, 
 without loss of generality, that $\tilde\eta = k + \eta$, $\eta>0$. Given any $\delta>0$, choose $N>0$
 such that  $N > \eta$ and $||f_{N-1}-f_N||_{\calC^{k +N}(\M)}<\eps_N <\delta$ (condition $(ii_N)$). Let $\nu, \mu > N$ with $\nu > \mu$. Then
 $$
 ||f_\nu - f_\mu||_{\calC^{k+\eta}(\M)} \le
 \sum_{l=0}^{\nu-\mu-1} ||f_{\mu+l} - f_{\mu+l+1}||_{\calC^\eta(\M_{k+\nu+l+1})} \le 
 \sum_{l=0}^{\nu-\mu-1} \frac{\eps_{\mu}}{2^l} \le
  \eps_{\mu} \sum_{l=0}^{\infty} \frac{1}{2^{l+1}} = \eps_{\mu} <\eps_N < \delta,
 $$
where we used the estimate
 $$
||f_{\mu+l} - f_{\mu+l+1}||_{\calC^{k+\eta}(\M)} \le
||f_{\mu+l} - f_{\mu+l+1}||_{\calC^{k+\mu+l+1}(\M)} \le
\eps_{\mu+l+1} < \frac{\eps_{\mu+l}}{2} \le \dots \le \frac{\eps_{\mu}}{2^{l+1}}.
$$
This shows that $\{f_j\}$ is a Cauchy sequence and therefore it converges in any $\calC^{\tilde\eta}(\M)$. Thus,
$f_\eps \in \calC^{\infty}(\M)$.

(2)  
Let $\nu\in \mathbb N$ be such that $||f _\infty -f_\nu||_{\calC^k(\M)} < \eps /2$. Then
$$
\begin{aligned}
||f_\infty -f||_{\calC^k(\M)} \le ||f_\infty - f_\nu||_{\calC^k} + ||f_\nu - f||_{\calC^k} \le \\
\eps/2 + \sum_{j=2}^\nu ||f_j - f_{j-1}||_{\calC^k} + ||f_1 - f||_{\calC^k}\le \eps/2 + \sum_{j=2}^\nu \eps / 2^j < \eps.
\end{aligned}
$$
This verifies condition (a). Conditions (b) holds by construction. To prove that $f_\eps(\M)$ is polynomially convex
we argue as follows. Note that by Lemma~\ref{L:2} any perturbation of $f$ in a neighbourhood of $f$ controlled by $c_1$ 
satisfies ~\eqref{e.cnb1}.  By the choice of $\eps_j$ in condition $(0_j)$ for all $j$ we conclude that 
$$
\widehat{K_1\cup f_\eps(\M)} \subset U_0 \cup f_\eps(\M) .
$$
The same argument can used to show that $\widehat{K_j \cup f_\eps(\M) }\subset U_{j-1} \cup f_\eps(\M)$ for any $j>0$, 
this again follows from the choice of $\eps_j$ in condition $(0_j)$. Since $U_j$ form a neighbourhood basis of $C$, we 
conclude that $\widehat{f_\eps(\M)} = f_\eps(\M)$. This completes the proof of the proposition.
\end{proof}

\section{A perturbation result for Coffman nondegenerate points}\label{S:C2}
In this section, we show that if an embedded $M\subset\Cn$ (in general position) contains only Coffman nondegenerate points outside of a polynomially convex compact set in $\Cn$, then it can be perturbed and made polynomially convex.   

Throughout this section, $\mathcal M$ is a closed smooth $m$-dimensional real manifold, endowed with a Riemannian metric, so that there are well-defined $\cont^k(\M;\Cn)$- and $\cont^k(\M;\M)$-norms for all $k\in\N$. Let $\cont^\infty_{\operatorname{emb}}(\M;\Cn)$ denote the space of smooth embeddings of $\M$ into $\Cn$, and $\cont^\infty_{\operatorname{diff}}(\M;\M)$ denote the space of smooth diffeomorphisms of $\M$. Given $p\in\M$ and $r>0$, $U_p(r)$ denotes the $r$-neighbourhood of $p$ in the metric fixed above. We use the notation as defined in Section~\ref{S:nondeg}.

First, we state a lemma that allows us to travel between the CR singularity sets of two nearby embeddings. 

\begin{lemma}\label{L:sing} 
Assume that $n\leq \frac{3}{2}m-1$. Let $f:\mathcal M\hookrightarrow\Cn$ be a smooth embedding such that $\mathcal S^f=\mathcal S_1^f$, and $\Phi_f$ is transverse to $\text{Hom}_\C^1(T^\C \mathcal M,\mathcal M\times\Cn)$. Fix an integer $k\geq 3$. We endow $\cont^\infty_{\operatorname{emb}}(\M;\Cn)$ with the $\cont^k$-norm, and $\cont^\infty_{\operatorname{diff}}(\M;\M)$ with the $\cont^{k-1}$-norm. Then, there exists a neighborhood $\mathfrak U\subset \cont^\infty_{\operatorname{emb}}(\M;\Cn)$ of $f$, and a continuous map $\theta:\mathfrak U\times\mathfrak U\rightarrow\cont^\infty_{\operatorname{diff}}(\M;\M)$ such that $\theta^g_h=\theta(g,h)$ satisfy the following conditions
	\begin{itemize}
\item [(a)] $\theta^g_h(\calS^{g})=\calS^{h}$,
\item [(b)] $\theta^g_h=\operatorname{id_\M}$, if $\calS^g=\calS^h$,
\item [(c)] $\theta^{g}_{h}=\theta^{p}_h\circ\theta^{g}_{p}$, for $g,h,p\in\mathfrak U$. In particular, $\theta^{g}_{h}=\theta^{p}_h$, if $\calS^g=\calS^p$. 
\end{itemize}  
there is a smooth diffeomorphism $\theta^f_g:\M\rightarrow\M$ such that
\end{lemma}
\begin{proof} This follows from the transversality of $\Phi_f$ to $\text{Hom}_\C^1(T^\C \mathcal M,\mathcal M\times\Cn)$ and the continuity of the map $f\mapsto \Phi_f$. 
\end{proof}

We now show that for a generic real-analytic embedding $f:\M\rightarrow\Cn$, a compactly-supported perturbation of $f|_{\calC_2}$ can be matched by a compactly-supported pertubation of $f$ on all of $\M$. Further, this can be done in a uniform manner for embeddings close to $f$, and over all points in a fixed compact subset of $\calC_2$. 

\begin{lemma}\label{L:AmbPertNonDeg} 
Let $f:\M\rightarrow\Cn$ be real-analytic and as in Lemma~\ref{L:sing}. Let $\mathcal L\subset \calC_2^f$ be a compact set. Fix an integer $k\geq 3$. Then, there is an $\wt\eps>0$ such that, the neighborhood $\mathfrak U$ granted by Lemma~\ref{L:sing} contains an $\wt\eps$-neighborhood (in the $\cont^k$-norm) of $f$ in $\cont^{\infty}(\M,\Cn)$, and for any $\eps\in(0,\wt\eps)$, there exists $\tau_\eps,\delta_\eps>0$ and $c_\eps>1$ such that the following holds for any $p\in\mathcal L$. Let $f': \M\hookrightarrow\Cn$ be a real-analytic embedding satisfying $\Vert f-f'\Vert_{\calC^k(\M)}<\eps$, and let $\theta^f_{f'}:\M\rightarrow\M$ be the corresponding diffeomorphism granted by Lemma~\ref{L:sing}. Then, for any smooth embedding $g:\mathcal S^{f'}\rightarrow\Cn$ such that 
\begin{itemize} 
\item [(a)] $g-f'|_{\calS^{f'}}$ is compactly supported in
$\mathcal S\cap \theta^f_{f'}(U_p(\tau_\eps))$,
\item [(b)] $\left\Vert g-f'|_{\calS^{f'}}\right\Vert_{\calC^k(\calS^g;\Cn)}<\delta_\eps$,
\end{itemize}
there is a smooth 
embedding $G:\mathcal M\hookrightarrow\Cn$ such that 
\begin{itemize}
\item [(i)] $G-f'$ is compactly supported in $\theta^f_{f'}(U_p(\tau_\eps))$,
\item [(ii)] $G|_{\calS^{f'}}=g$,
\item [(iii)] $\mathcal S^G=\mathcal S^{f'}$,  
\item [(iv)] $\left\Vert G-f'\right\Vert_{\calC^k(\mathcal M)}<c_\eps\delta_\eps$.
\end{itemize}
\end{lemma}

\begin{proof} We produce $\tau,\delta>0$ and $c>1$, corresponding to $f'=f$ and fixed $p\in\mathcal L$ so that (i)-(iv) hold for a $g$ satisfying (a) and (b), and observe that all the steps in this proof depend continuously on the choice of $f'$ and $p$.

Since $p\in\mathcal C_2$, by Proposition~\ref{P:Coff}, there is a neighborhood $V_p$ of $f(p)$ in $\Cn$ and a biholomorphism $Q:V_p\rightarrow \mathbb B^{m}(0,\eta)\times \mathbb B^{2n-m}(0,\eta)$, for some $\eta\in(0,1)$, such that $Q(p)=0$, and
\beas
		Q(M\cap V_p)=\left\{z\in \mathbb B^{m}(0,\eta)\times \mathbb B^{2n-m}(0,\eta): \ \ 
\begin{aligned}
	&  y_s=0,\ 2\leq s\leq m-1,\\
	& z_{m}=\overline z_1^2,\\
	&z_{m+1}=|z_1|^2+\overline z_1(x_2+ix_3),\\
	&z_u
		=\overline z_1 (x_{2(u-m)}+ix_{2(u-m)+1}),\ m+2\leq u\leq n\\
\end{aligned}
\right\}.
\eeas

Thus, for the rest of the proof, it suffices to assume that $U_p(\tau)=\mathbb B^m(0,\eta)\subset\rl^m$, which has coordinates $\mathbf x=(x_1,y_1,x_2,...,x_{m-1})$, and $f=(f_1,...,f_n):\mathbb B^m(0,\eta)\rightarrow \Cn$ is given by   
\beas 
f_1(\mathbf x)&=&x_1+iy_1,\\
f_s(\mathbf x)&=& x_s,\quad s=2,...,m-1,\\
f_m(\mathbf x)&=&(x_1 -iy_1)^2,\\
f_{m+1}(\mathbf x)&=&(x_1 - iy_1)(x_1+iy_1 + x_2 + i x_3),\\
f_u(\mathbf x)&=&(x_1-iy_1)(x_{2(u-m)}+ix_{2(u-m)+1}),\quad m+2\leq u\leq n.
\eeas 
Referring to \eqref{E:CRS}, we obtain that 
\bes \calS\cap U_p(\tau) =\left\{\mathbf x\in \mathbb B^m(0,\eta):
x_1=y_1=x_2=\cdots=x_{2(n-m)+1}=0\right\}. 
\ees
Let $\mathbf x'=(x_{2n-2m+2},...,x_{m-1})$, and  $f|_{\mathcal S}$ be denoted by $f_{\calS}$. Then, 
	\bes
		f_{\calS}:(0,\mathbf x')\mapsto (0,...,0), \quad \mathbf x'\in\mathbb B^{3m-2n-2}(0,\eta). 
	\ees
Given $h\in\cont^{\infty}_c(\mathbb B^{3m-2n-2}(0,\eta);\Cn)$, let $g(0,\mathbf x')=f_{\calS}(0,\mathbf x')+h(\mathbf x')$, $\mathbf x'\in \mathbf B^{3m-2n-2}(0,\eta)$ (and $f_{\calS}$ outside this ball). Assume that 
$\Vert h\Vert_{\calC^k(\mathbb B^{3m-2n-2}(0,\eta))}<\delta$,
where $\delta$ is chosen so that $g(\calS\cap U_p(\tau))\subset \mathbb B^{2n}(0,\eta)$. Then, $h$ takes the form 
	\bes
		h(\mathbf x')=\left(\phi^1(\mathbf x')+i\psi^1(\mathbf x'),...,\phi^n(\mathbf x')+i\psi^n(\mathbf x')\right),
	\ees
for some $\phi^j,\psi^j\in\calC_c^\infty\left(\mathbb B^{3m-2n-2}(0,\eta)\right)$, with $\Vert\phi^j\Vert_{\calC^k},\Vert\psi^j\Vert_{\calC^k}<\delta$.

Now, let $G=(G_1,...,G_n):\mathbb B^{m}(0,\eta)\rightarrow\Cn$ be given by 
\bes
		G:\mathbf x\rightarrow f(\mathbf x)+\left(\phi^1(\mathbf x')+i\psi^1(\mathbf x'),...,\phi^n(\mathbf x')+i\psi^n(\mathbf x')\right).
\ees
Claims (ii) and (iv) clearly hold for $G$ (for an appropriate choice of $c>1$). To compute $\calS^{G}$ (to confirm (iii)), we seek points in $\mathbf x\in \mathbb B^{m}(0,\eta)$ at which the rank of the matrix $dG_\C(\mathbf x)$ is $m-1$. After some elementary row reductions, the matrix becomes 
For sufficiently small $\eps_p$, $\mathbf I_{3m-2n-2}+\mathbf E$ is invertible, and thus $\text{rank}DG(\mathbf x)\geq m-2$ for all $\mathbf x\in V_p$. Thus, we seek those $\mathbf x\in V_p$ for which $\text{rank}_{\C!}dG_\C(\mathbf x)=m-1$. After some simple row reductions (and shrinking $\delta$, if necessary), the matrix $dG_\C(\mathbf x)$ becomes
\bes
\left(\begin{array}{c c|c c c}
1 & i & 0 & \cdots                        & 0 \\
& & & &  \\
\hline
& & & &\\
0 & 0 & & &  \\
\vdots & \vdots &   & \mathbf I_{m-2} &    \\
0 & 0 & & &  \\
& & & &\\
\hline
& & & & \\
2(x_1-iy_1)& 
-i2(x_1-iy_1)
& & &  \\
2x_1+x_2+ix_3 & 
-2y_1-i(x_2+ix_3)& 
 &  &  \\
x_4+ix_5& 
-i(x_4+ix_5)
&  & \mathbf 0_{(n-m+1)\times(m-2)}  & \\
\vdots & \vdots & & \\
x_{2n-2m}+ix_{2n-2m+1}& 
-i(x_{2n-2m}+ix_{2n-2m+1})
&  &\\
\end{array}\right),
\ees
which has rank $m-1$ precisely on the set 
	\bes
	\calS\cap U_p(\tau) =\left\{\mathbf x\in \mathbb B^m(0,\eta):
x_1=y_1=x_2=\cdots=x_{2(n-m)+1}=0\right\}.
	\ees
Finally, for (i), we note that $G$, as defined, is not compactly supported in the directions normal to the $\mathbf x'$ coordinates. However, $G$ is totally real off of the $\mathbf x'$-axes, and thus, we can replace $G$ by 
	\bes
		G(\mathbf x)\mapsto f(\mathbf x)+\chi(\mathbf x'')h(\mathbf x'), 
	\ees 
where $\mathbf x''=(x_1,y_1,x_2,...,x_{2(n-m)}+1)$ and $\chi$ is a radially symmetric cutoff function that is a constant $1$ in a neighborhood of $\mathbf x''$, and is chosen such that $G$ remains totally real off of the $\mathbf x'$-axes. 

\end{proof}


Before we proceed to tackle global perturbations, we prove a result that will be invoked multiple times. 

\begin{lemma}\label{L:PCNbhd} 
Let $T\subset\Cn$ be a closed totally real submanifold. Let $K\subset T$ be a compact subset that is polynomially convex in 
$\Cn$. Then, there is a bounded neighbourhood $U$ of $K$ in $\Cn$ such that $\overline{U}\cap T$ is polynomially convex 
in $\Cn$ and is a manifold with smooth boundary.
\end{lemma}
\begin{proof}
Since $K$ is polynomially convex in $\Cn$, there exists a smooth plurisubharmonic function $\rho$ on $\C^n$ such that $\rho$ vanishes on $K$, and is positive and strictly plurisubharmonic on $\C^n \setminus K$. For $z\in\Cn$, let $\delta(z)$ denote the squared Euclidean distance of $z$ from $T$. Then, $\delta$ is a strictly plurisubharmonic function on some neighbourhood $W$ of $T$. Given $\eps>0$, let 
\beas 
U_\eps&=&\{z\in\Cn:\rho(z)<\eps\},\\
W_\eps&=&\{z\in\Cn:\delta(z)<\eps\}.
\eeas
Choose $\eps_1, \eps_2 >0$ such that $U_{\eps_2}\Subset W_{\eps_1}\Subset W$. 
Let $\chi\geq 0$ be a cutoff function which is 1 on $ W_{\eps_1}$ and vanishes outside $W$. 
Let $R>0$ be large enough so that $\chi \delta < R(\rho -\eps_2)$ on $\C^n \setminus  W_{\eps_1}$. 
Now consider the function
$$
\sigma (z) = \max \{R(\rho(z) - \epsilon_2), \chi(z) \delta(z)\},\qquad z\in\Cn.
$$ 
By construction, $\sigma$ is a continuous function on $\C^n$ that vanishes on  
$\overline {U_{\eps_2}}\cap T= \{ z\in T : \rho(z) \le \eps_2\}$, 
and is positive everywhere else. Furthermore, $\sigma$ is plurisubharmonic on $\C^n$. Indeed, on $W_{\eps_1}$, it is 
the maximum of two plurisubharmonic functions, and outside a compact subset of  $W_{\eps_1}$, it agrees with 
$R(\rho-\eps_2)$. Thus, 
we obtain that $\overline U_{\eps_2}\cap T$ is polynomially convex. 

By Sard's theorem, for almost all $c\in \R$, the level set $\{(\rho|_T)^{-1}(c)\}$ is a smooth submanifold of $T$. 
Therefore, since $K$ is a compact subset of $\overline U_{\eps_2}\cap T$, there exists a choice of $\eps_3 >0$ such that 
$$
\overline U_{\eps_3}\cap T \Subset \overline U_{\eps_2}\cap T ,
$$
and $\{\rho = \eps_3\}\cap T$ is a smooth manifold. The set $\overline U_{\eps_3}$ is polynomially convex, and therefore, 
$$
\overline U_{\eps_3} \cap T = \overline U_{\eps_3} \cap (\overline U_{\eps_2}\cap T)
$$ 
is polynomially convex as the intersection of two polynomially convex compacts (see~\cite[p.~5]{St}). 
Thus, $U=U_{\eps_3}$ is the required neighbourhood of $K$.    
\end{proof}

The next proposition shows that if $M\subset\Cn$ in general position has only order-one CR singularities, and its set of Coffman degenerate singularities admits a polynomially convex neighbourhood within its set of CR singularities, then $M$ can be perturbed to make its entire set of CR singularities polynomially convex.

\begin{prop}\label{P:Amb}
Let $f:\M\rightarrow\Cn$ be as in Lemma~\ref{L:sing}. Fix an integer $k\geq 3$. Let $\calK\subset \calS^f$ be a compact set such that $\calS^f\setminus\calK\Subset \calC_2^f$, and $K=f(\calK)$ is a polynomially convex 
compact subset of $\Cn$. Then, given $\eps>0$, there is an embedding 
$F: \M\hookrightarrow\Cn$ such that
\begin{itemize}
\item [(1)] $\Vert f-F\Vert_{\calC^k(\M;\Cn)}<\eps$,
\item [(2)] the set of CR-singularities of $F(\M)$ is polynomially convex in $\Cn$. 
\end{itemize} 
\end{prop}

\begin{proof} After a small perturbation, we may assume that $f:\M\rightarrow\Cn$ is real-analytic. Denote 
$\calS^f$ by $\calS$ and $\calC_2^f$ by $\calC_2$. Let $\mu = \dim_{\rl\!} \calS$. As per our convention, set $M = f(\M)$, $S=f(\calS)$, $C_2=f(\calC_2)$, and $K=f(\calK)$.

Let 
$\mathcal L=\overline{\calS\setminus\calK}$, 
and $\wt\eps>0$ as granted by Lemma~\ref{L:AmbPertNonDeg} corresponding to this choice of $\mathcal L$. Let $\eps\in(0,\wt\eps)$. Let $\tau_\eps,\delta_\eps>0$ and $c_\eps>1$ be as granted by Lemma~\ref{L:AmbPertNonDeg} corresponding to this choice of $\eps$. 
The idea of the proof is to incrementally increase the size of the 
polynomially convex compact set in the set of CR singularities by adding to $\calK$ a smooth image $\calR$ in $\calS$ of a standard cube $I^\mu=[0,1]^\mu \subset \R^\mu$ 
and then producing a smooth perturbation $g$ 
of $f|_{\calS}$ so that $g(\calK \cup \calR)$ is polynomially convex, followed by a matching smooth perturbation $f^\sharp$ of $f$ on $\M$. This process is then repeated for (a real-analytic perturbation of) $f^\sharp$. The size of $\calR$ will depend on $\tau_\eps$, and will be such that finitely many such cubes cover $\mathcal L$. After applying this procedure finitely many times, we obtain a perturbation of $f$ 
for which the whole set of CR-singularities is polynomially convex. A detailed construction is as follows.

Let $\calK\subset\calS$ be as given. Let $\calR=\phi(I^\mu) \subset S$, where $\phi$ is a 
diffeomorphism from a neighbourhood of $I^\mu$ in $\R^\mu$ onto its image in $\calS$. Assume that $\calR\cap \calK \ne \varnothing$. Set $R=f(\calR)$. 
The principal step is to obtain a polynomially convex perturbation of $K \cup R$. For this, we induct over the skeleta of $\calR$. Let $\calR_k$ denote the $k$-dimensional skeleton of $\calR$, for $k=0,1,...,\mu$. Set $R_k=f(\calR_k)$, $k=0,1,...,\mu$. 

As a base of the induction, note that since $R_0$ is a finite set of points, by~\cite[Lemma 2.1]{GS1}, the set $K \cup R_0$ is polynomially convex. By 
Lemma~\ref{L:PCNbhd}, there exist compact 
polynomially convex neighbourhoods $V_0\Subset V_0'\subset S$ of $K \cup R_0$ that are smooth manifolds with boundary. Set $Q_1=R_1\setminus V_0$ and ${Q_1'}=R_1\setminus V_0'$. Then, $Q_1'\subset Q_1\subset R_1$ are smooth (totally real) 
manifolds with boundary of dimension $1$ lying entirely in $C_2$, and such that 
\bes 
K \cup R_0 \cup R_1\subset V_0 \cup Q_1= V_0' \cup Q_1'.
\ees
Set $g_0=f$. Set $\mathcal Q_1=g_0^{-1}(Q_1)$ and $\mathcal Q_1'=g_0^{-1}(Q_1')$. Note that $\mathcal Q_1'\subset\mathcal Q_1\subset \mathcal R_1$. 

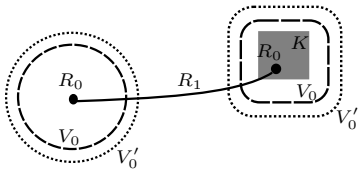
\begin{figure}[H]
\begin{overpic}[grid=false,tics=10,scale=0.3]{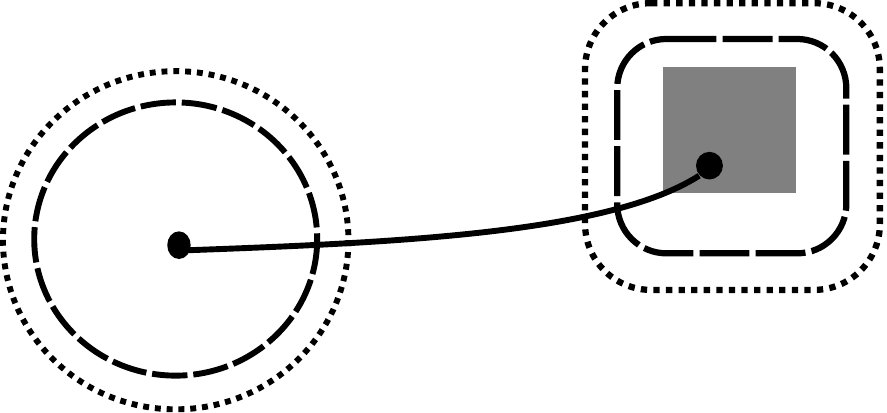}
\begin{tiny}
\put(20,29){$R_0$}
\put(20,8){$V_0$}
\put(42,0){$V_0'$}
\put(95,40){$R_0$}
\put(110,25){$V_0$}
\put(125,15){$V_0'$}
\put(108,43){$K$}
\put(65,29){$R_1$}
\end{tiny}
\end{overpic}
\smallskip 
\caption{The grey region is $K$; the dashed and dotted lines bound $V_0$ and $V_0'$, respectively.}\label{F:base}
\end{figure}
\vspace{-10pt}
By \cite[Theorem 1.4]{ArWo19}, applied to $\mathcal Q_1'$ as $M$ and $V_0$ as $K$, for any $\delta>0$, there exists a perturbation of $g_0|_{\mathcal Q_1}$ of the form $g_0+h$, so that $h:\mathcal Q_1\rightarrow g_0(\calC_2)$ 
	\begin{itemize}
	\item [(a)] is compactly supported on $\mathcal Q_1$, 
	\item [(b)] has $\delta$-small $\calC^k$-norm,
	\item [(c)] vanishes on $\mathcal Q_1\setminus \mathcal Q_1'$, 
	\end{itemize}
and $V_0\cup (g_0+h)(\mathcal Q_1)$ is polynomially convex. We claim that given $\delta >0$ there exists a perturbation $g_1: \calS \to \C^n$ of $g_0$ satisfying
\begin{itemize}
\item [(1)] $g_1-g_0$ is compactly supported in the $\delta$-neighbourhood of $\mathcal Q_1$ in $\calS$,
\item [(2)] $\Vert g_1-g_0\Vert_{\calC^k(\calS)}<\delta$
\item [(3)] $g_1\left(\calK\cup \mathcal R_0\cup\mathcal R_1\right)$ is polynomially convex in $\Cn$. 
\end{itemize}
Using a partition of unity, we may assume that the compact support of $h$ that lies in $\mathcal Q_1\subset \calS$, is in fact contained in a 
coordinate chart of $\calS$, which does not intersect $\calK\cup \mathcal  R_0\cup(\mathcal R_1\setminus \mathcal Q_1)$. In this chart, we may assume that $\calS = \R^{1}_x \times \R^{\mu-1}_y$, $\mathcal Q_1=\R^{1}_x\times\{0\}$, $h=h(x)$, and $\R^1_x\times [-\eta,\eta]^{\mu-1}$ is contained in the $\delta$-neighbourhood of $\mathcal Q_1$ in $\mathcal  S$.  
Let $\sigma$ be a radial bump function supported on $[-\eta,\eta]^{\mu-1}$ with $\sigma(0)=1$. Choose $A>0$ such that $\Vert\sigma\Vert_{\calC^k(\calS)} < A$, and choose $\delta_1 < \delta/A$. Let $h$ be the map produced above, but for $\delta=\delta_1$. Then, the function $h(x)\sigma(y)$ has compact support in the coordinate chart of 
$\calS$, and can be extended as 0 to the rest of $\calS$. We denote this extension by $h_0$. It agrees with $h$ on $\mathcal Q_1$, vanishes on $\calK\cup \mathcal  R_0\cup(\mathcal R_1\setminus \mathcal Q_1)$, 
and satisfies $\Vert h_0\Vert_{\calC^k(\calS)}=\Vert h(x) \sigma(y)\Vert_{\calC^k(\calS)} <\delta$. This implies that $g_1 = g_0+ h_0$ is a perturbation of 
$\calS$ that satisfies (1)-(3). This proves the claim.
 
We repeat this procedure inductively. Finally, we have that for any $\delta>0$, there exists a smooth embedding $g:\calS\rightarrow\Cn$ such that $f-g$ is compactly supported on the $\delta$-neighbourhood of $\mathcal R$, $f|_K=g|_K$, $\Vert f-g\Vert_{\calC^k(\calS)}<\delta$, and $g(\calK\cup\calR)$ is polynomially convex in $\Cn$. 

We now combine this procedure with Lemma~\ref{L:AmbPertNonDeg}. 
Let $\eps >0$ be as given. Set $\delta_1=\min\{\delta_\eps,\eps/4c_{\eps}\}$. Let 
$p_1\in \calK \setminus \operatorname{int}{\calK}$ so that $p_1 \in \calC_2$. Let $\calR^1$ be a diffeomorphic image of the cube $I^\mu \subset \R^\mu$ in $\calS$ such that the $\delta_1$-neighbourhood of $\calR^1$ is contained in $U_{p_1}(\tau_\eps)$ and $p_1\in\calR^1$. By the principal step outlined above, there exists a perturbation $g_1: \calS\to \C^n$ of $f|_{\calS}$ such that $g_1$ coincides with $f$ outside a compact set in $\calS\cap U_{p_1}(\tau_\eps)$, satisfies $\Vert f-g_1\Vert_{\calC^k(\calS)}<\delta_1$, and $g_1(\calK\cup\calR^1)$ is polynomially convex. By Lemma~\ref{L:AmbPertNonDeg}, (applied to $f'=f$ itself),
there exists a smooth perturbation $G_1: \M \to \C^n$ of $f$,  which agrees with $f$ outside a compact subset of $U_{p_1}(\tau_\eps)$, coincides with $g_1$ on $\calS$, satisfies $\calS^{G_1}=\calS^{f}=\calS$, and 
	\bes	
		\Vert f-G_1\Vert_{\calC^k(\M)}<c_\eps\delta_1\leq\frac{\eps}{4}.
	\ees	
We wish to continue by induction, but $G_1$ itself may not be real-analytic. We will replace $G_1$ by a carefully chosen real-analytic approximant of $G_1$ as follows. Recall that $P=G_1(\calK\cup \calR^1)$ is a polynomially convex compact subset of the totally real manifold $G_1(\calS)$. By Lemma~\ref{L:PCNbhd}, there is a bounded totally real manifold with smooth boundary $\Omega\subset G_1(\calS)$ that contains $P$ and whose closure is polynomially convex. By \cite[Theorem 1]{LW}, there is a $\delta_{\Omega}>0$ such that any smooth perturbation of $G_1:G_1^{-1}(\Omega)\rightarrow\Cn$ that is $\delta_{\Omega}$ close to it in the $\cont^{k-1}$-norm (on $G_1^{-1}(\Omega)$) must be a totally real and polynomially convex embedding of $G_1^{-1}(\Omega)$. Recall that $\mathfrak U\subset\cont^\infty_{\operatorname{emb}}(\M;\Cn)$ is the neighborhood of $f$ as granted by Lemma~\ref{L:sing}. Since $(f,G_1)\in\mathfrak U\times\mathfrak U$ (by our choice of $\wt\eps$) and $\calS^f=\calS^{G_1}$ (by construction), there exists an $\eps_\Omega>0$ such that 
\bes
	\text{if }\Vert G_1-h\Vert_{\cont^k}<\eps_\Omega,\quad  
\text{then }
	\Vert G_1-h\circ \theta^{f}_{h}\Vert_{\cont^{k-1}}<\delta_\Omega.  
\ees
Here, we have used that $\theta^{G_1}_h=\theta^f_{h_1}$ (see (c) in Lemma~\ref{L:sing}). Now, let $f_1:\M\rightarrow\Cn$ be a real-analytic embedding of $\M$ such that 
	\bes
		\Vert f_1-G_1\Vert_{\cont^k(\M)}<\min\left\{\eps_\Omega,\frac{\eps}{4}\right\}.
	\ees
Then, the diffeomorphism $\theta^{f}_{f_1}=\theta^f_{f_1}:\calS^f\rightarrow\calS^{f_1}$ satisfies $\Vert G_1-f_1\circ \theta^{f}_{f_1}\Vert_{\cont^{k-1}(\M)}<\delta_{\Omega}$. Thus $\left(f_1\circ\theta^{f}_{f_1}\right)(\calK\cup\calR^1)$ is polynomially convex in $\Cn$. We have produced a real-analytic embedding $f_1:\M\rightarrow\Cn$ such that the following hold:
	\begin{itemize}
\item $\Vert f-f_1\Vert_{\cont^k(\M)}<\eps/2$,
\item $(f_1\circ \theta_1)(\calK\cup\calR^1)$ is polynomially convex in $\Cn$, for some diffeomorphism $\theta_1:\calS^{f}\rightarrow\calS^{f_1}$.
\end{itemize}

We demonstrate the induction procedure by outlining the next step. Set $\delta_2=\min\{\delta_\eps,\eps/8c_\eps\}$. Let $\delta_2^*>0$ be such that 
	\be\label{E:theta}
		\text{if}\left\Vert f_1\circ \theta^{f}_{f_1}-g_2\circ \theta^{f}_{f_1}\right\Vert_{\cont^k(\calS^f)}<\delta_2^*,\quad \text{then}
			\Vert f_1-g_2\Vert_{\cont^k(\calS^{f_1})}<\delta_2. 
	\ee
Let $p_2\in (\calK\cup\calR^1)\setminus\operatorname{int}(\calK\cup\calR^1)$. Let $\calR^2$ be a cube in $\calS$ such that the $\delta_2^*$-neighbourhood of $\calR^2$ is contained in $U_{p_2}(\tau_\eps)$ and $p_2\in\calR^2$. As before, by the principal step outlined above, there is a perturbation $\wt g_2: \calS\to \C^n$ of $\wt f_1=(f_1\circ \theta^{f}_{f_1})|_{\calS}$ such that $\wt g_2$ coincides with $\wt f_1$ outside a compact set in $\calS\cap U_{p_2}(\tau_\eps)$, satisfies $\Vert \wt f_2-\wt g_2\Vert_{\calC^2(\calS)}<\delta^*_2$, and $\wt g_2(\calK\cup\calR^1\cup\calR^2)$ is polynomially convex. Set $g_2=\wt g_2\circ {\theta^{f}_{f_1}}^{-1}$ on $\calS^{f}$. Since $\Vert f-f_1\Vert_{\calC^2(\calS)}<\eps$, we invoke Lemma~\ref{L:AmbPertNonDeg} for $f'=f_1$ to obtain a smooth perturbation $G_2:\M\rightarrow\Cn$ of $f_1$, that coincides with $g_2$ on $\calS^{f_1}$, satisfies $\calS^{G_2}=\calS^{f_1}$, and 
	\bes	
		\Vert f_1-G_2\Vert_{\calC^k(\M)}<c_\eps\delta_2\leq\frac{\eps}{8}.
	\ees	
To replace $G_2$ with a real-analytic approximant, we repeat the idea above for $P=(G_2\circ\theta^{f}_{f_1})(\calK\cup\calR^1\cup\calR^2)$, to produce a real-analytic embedding $f_2:\M\rightarrow\Cn$ so that $\Vert f_1-f_2\Vert_{\cont^k(\M)}<\eps/4$, and $(f_2\circ\theta_2):(\calK\cup\calR^1\cup\calR^2)$ is polynomially convex for some diffeomorphism $\theta_2:\calS^f\rightarrow\calS^{f_2}$.

Repeating this procedure $N$ times, we obtain an embedding $f_N:\M\rightarrow\Cn$ and a diffeomorphism $\theta_N:\calS^f\rightarrow\calS^{f_N}$ such that $(f_N\circ\theta_N)(\calK\cup\calR^1\cup\cdots\cup\calR^N)$ polynomially convex in $\Cn$, and 
\bes
\Vert F-f\Vert_{\calC^2({\M})}\le\Vert f_{N} - f_{N-1}\Vert_{\calC^2({\M} )} +\cdots +\Vert f_{1} - f\Vert_{\calC^2({\M} )} <
\sum_{j=1}^\infty \eps/2^j = \eps.
\ees
By the compactness of $S$, we can ensure that there is some $N$ and choice of cubes $\calR^j$'s so that $\calK\cup\calR^1\cup\cdots\cup\calR^N=\calS$. This proves the proposition.
\end{proof}

\section{Proof of the main result}\label{S:main}

The idea of the proof is to progressively make the strata of the CR-singular set of $M=f(\M)$ polynomially convex using 
the results of Section~\ref{S:pert}. We assume that the given $f: \M\hookrightarrow\Cn$ is in general position so that the conclusion of Proposition~\ref{P:PertII} holds for $f$. By the dimension count in Proposition~\ref{P:PertII}, $W_0=C_0=\emptyset$, and thus $S:=f(\calS)=f(\calS_1)$ is a totally real submanifold in $\Cn$. The proof splits into the following three cases.
\begin{itemize}
\item [(i)] $m\equiv 0,1 \pmod 4$. In this case, $S=C_2$. That is, all CR-singularities of $M$ are of Coffman nondegenerate type. 
\item [(ii)] $m\equiv 2 \pmod 4$. In this case, $S=C_2\sqcup C_1$ and $\dim_{\rl\!} C_1=0$. That is, all but finitely many CR singularities of $M$ are of Coffman nondegenerate type. 
\item [(iii)] $m\equiv 3 \pmod 4$. In this case, $S=C_2\sqcup (C_1\cap W_2)\sqcup (C_1\cap W_1)$, $\dim_{\rl\!} C_1=1$ and $\dim_{\rl\!} W_1=0$. That is, the set of Coffman degenerate CR singularities is a disjoint union of finitely many smooth simple closed
curves, of which all but finitely many points are Webster nondegenerate. 
\end{itemize}

We now prove the theorem case by case. 

\medskip

\noindent {\bf Case 1.} Assume that $\dim_{\rl\!}M=m\equiv 0,1 \pmod 4$. Since $S=C_2$ in this case, the result follows from Proposition~\ref{P:Amb} and Lemma~\ref{l.pertlemma}.

\medskip

\noindent {\bf Case 2.} Assume that $\dim_{\rl\!}M=m\equiv 2 \pmod 4$. Since $C_1$ is a $0$-dimensional submanifold of $S$, it is a finite union of points, and thus, is polynomially convex. By Lemma~\ref{L:PCNbhd}, there is a neighbourhood $U$ 
of $C_1$ in $\Cn$ such that $K=\overline{U}\cap S$ is polynomially convex. The result now follows from Proposition~\ref{P:Amb} (applied to this choice of $K$) and Lemma~\ref{l.pertlemma}.

\medskip

\noindent{\bf Case 3.} Assume now that $\dim_{\rl\!}M=m\equiv 3 \pmod 4$. Suppose, for the moment, that $C_1$ is a polynomially convex subset of $\Cn$. Then, by Lemma~\ref{L:PCNbhd}, there is a neighbourhood $U$ of $C_1$ in $\Cn$ 
such that $K=\overline{U}\cap S$ is polynomially convex. By Proposition~\ref{P:Amb}, there is a small perturbation $F$ of $f$ so that the set of CR singularities of $F$ is a polynomially convex compact subset of $\Cn$. The desired perturbation is then obtained by applying Lemma~\ref{l.pertlemma} to $F$.

In view of the above argument, it suffices to produce a $
\calC^3$-small perturbation of $f$ that is in general position (so that the conclusion of Proposition~\ref{P:PertII} holds) and for which the set $C_1$ (of Coffman degenerate points) is a polynomially convex compact set in $\Cn$. In the discussion below we use the fact that a disjoint union of finitely many smooth closed curves in $\C^n$ is either polynomially convex, or the nontrivial portion of its polynomially convex hull is a complex 1-dimensional variety in the complement of this union, see, e.g., Stout~\cite[Ch. 3]{St}.

\begin{lemma}\label{L:tang} Let $\Gamma$ be a disjoint union of smooth simple closed curves $\gamma_1,...,\gamma_k$ in $\Cn$, which is polynomially convex in $\Cn$. Let $\gamma$ be a smooth simple closed curve that is disjoint from $\Gamma$, and such that $\Gamma\cup\gamma$ is not polynomially convex. Assume that some subarc $\gamma^\sharp\subset\gamma$ is real-analytic. Then, there is a dense set of points $p\in\gamma^\sharp$ with the following property. There exists a $1$-dimensional complex subspace $\ell_p$ of $\Cn$ and neighbourhood $U_p$ of $p$ in $\Cn$ such that: for any compactly supported smooth map $h:U_p\rightarrow\Cn$ satisfying 
\begin{itemize}
\item [(a)] $h(p)=0$, 
\item [(b)] $F=\operatorname{Id}+h$ is a diffeomorphism such that $F(\gamma^\sharp\cap U_p)\subset U_p$. 
\item [(c)] $\Gamma\cup F(\gamma)$ is not polynomially convex, 
\end{itemize}
we have that, the curve $\wt\gamma=F(\gamma)$ is tangential to $\ell_p$ at $p$. 
\end{lemma}
\begin{proof}  Since $K=\Gamma\cup \gamma$ is not polynomially convex in $\Cn$, there exists a purely $1$-dimensional analytic variety $X\subset\Cn\setminus K$ such that 
	\bes
		\widehat K=X\cup K.
	\ees
Due to the regularity of $\gamma$ and analyticity of $\gamma^\sharp$, there is a compact set of zero length 
$\varsigma\subset\gamma$ such that for each $p\in(\gamma^\sharp\setminus \varsigma)$ there is some 
neighbourhood $U_p\subset U$ of $p$, so that $X\cup(\gamma^\sharp\cap	U_p)$ extends past $\gamma^\sharp\cap U_p$ as a complex curve $Y_p$ that is regular along $(\gamma^\sharp\setminus \varsigma)\cap U_p$, and in particular, at $p$. Fix such a $p$, and let $\ell_p$ denote the (complex) tangent space of $Y_p$ at $p$. 

Now, let $h:U_p\rightarrow\Cn$ be a compactly supported map satisfying (a)-(c). Then, $\wt\gamma=F(\gamma)$ differs from $\gamma$ only in $U_p$ by a smooth arc $\sigma$ such that $(\wt\gamma\setminus\sigma)\cap U_p$ contains a real-analytic subarc, say $\tau$. Since $K=\Gamma\cup\wt\gamma$ is not polynomially convex, there is a one-dimensional variety $\wt X\subset\Cn\setminus\wt\gamma$ such that 
	\bes
		\widehat {\wt K}= \widetilde X\cup {\wt K}.
	\ees
By the regularity of $\wt\gamma$, there is a compact set of zero length $T\subset\wt\gamma$ such that $\wt X\cup\wt\gamma$ is a smooth manifold with boundary near each point in $\wt\gamma\setminus T$. In particular, this is true along some subarc $\tau'$ of $\tau\setminus(S\cup T)$. By the reflection principle, $\wt X\cup\tau'$ extends analytically past $\tau'$ as a complex curve, say $Z_{\tau'}$, that is regular along $\tau'$. Let $\wt X_1$ be the global branch of $\wt X$ that continues as $Z_{\tau'}$ past $\tau'$. But $Y_p$ also contains the arc $\tau'$ in its regular set, as $\tau'\subset (\gamma^\sharp\setminus S)\cap U_p$. Thus, by the identity principle for analytic curves, it must be that $Y_p$ is the analytic continuation past $p$ of $\wt X_1$. However, since $\Gamma$ is polynomially convex, every global branch of $\wt X$ must contain $\gamma$ in its closure, and in particular, the real-analytic subarc $\tau'$, which implies that $\wt X$ has only one global branch. Thus, the regular curve $\wt \gamma$ lies in $Y_p$ and must be tangential to $\ell_p$ at $p$.
\end{proof}

To construct a perturbation of $M=f(\M)$ for which the set $C_1$ is polynomially convex, we argue by induction 
on the number of connected components of $C_1$. Let $C_1$ be the disjoint union of $r$ smooth simple closed real curves $\gamma_1,...,\gamma_r$ in $M$. For any $\calC^3$-small perturbation $\wt f$ of $f$, the Coffman degenerate set $\wt C_1$ of $\wt f(\M)$ is also a disjoint union of simple closed curves $\wt\gamma_1,...,\wt\gamma_r$, where each $\wt\gamma_k$ is a $\calC^1$-small perturbation of $\gamma_k$. Suppose, for some $k\in\{1,...,r\}$,  $\gamma_1\cup\cdots\gamma_k$ is polynomially convex, but $\gamma_1\cup\cdots\gamma_{k+1}$ is not. We now show that there exists a $\calC^3$-small perturbation $\wt f$ of $f$, supported in a small neighbourhood of a 
Webster nondegenerate point $p\in\gamma_{k+1}$ so that  $\wt\gamma_1\cup\cdots\cup\wt\gamma_{k}\cup\wt\gamma_{k+1}=
\gamma_1\cup\cdots\cup\gamma_{k}\cup\wt\gamma_{k+1}$ is polynomially convex.

We denote $\gamma_{k+1}$ by $\gamma$. Let $p\in\gamma$ be a Webster non-degenerate point (recall that there are only finitely many Webster degenerate points in this dimensional setting). Locally, we may assume that $M$ is given by \eqref{E:PN},\eqref{E:PN1}, \eqref{E:PN2}, and \eqref{E:emb}. Using cutoff functions, we may get rid of the $O(4)$-terms in a small neighbourhood of $p$ in $\Cn$. This can be achieved via a $\calC^3$-small perturbation of $f$ that is compactly supported in a neighbourhood of $p$ (which, by Proposition~\ref{P:PertII} preserves all the features of the stratification of $\calS$). Then, a subarc $\gamma^\sharp$ of the perturbed $\gamma$ (henceforth denoted by $\gamma$) is real-analytic. We may now apply Lemma~\ref{L:tang} to obtain a dense set of points in $\gamma^\sharp$ that have the property given in the statement of the lemma. Without loss of generality, we may assume that $p$ is such a point, and $\ell_p$ is the one-dimensional complex subspace of $\Cn$ granted by Lemma~\ref{L:tang}. Recall the local expression for $\gamma$ at $p=0$ from Proposition~\ref{P:transC1}. The rest of the discussion will take place in these local coordinates. 

Let $\sigma$ denote the projection to the graphing coordinates $\mathbf x=(x_1,y_1, x_2, x_{m-1})$  
of the curve $\gamma$. Then, by Proposition~\ref{P:transC1}, $\sigma=\{\mathbf x:\mathscr C(\mathbf x)=0\}$, where $\mathscr C$ is as in \eqref{E:defC1}, and the $(m-1)\times m$ matrix $d\mathscr C(0)$ has rank $m-1$. Since $\gamma$ is tangential to $\ell_p$ at $p$, we have that $\ker d\mathscr C(0)$ lies in $\Pi_{\mathbf x}\ell_p$ --- the projection of $\ell_p$ onto the $(x_1,y_1,...,x_{m-1})$ coordinates, which is at most $2$-dimensional. We will produce a compactly supported $\calC^3$-small perturbation of $M$ (by perturbing the coefficients of the order $3$ terms in \eqref{E:PN1} and \eqref{E:PN2}) such that the perturbed $\gamma$, say $\wt\gamma$, passes through $p$, is regular there, but its not tangential to $\ell_p$ at $p$. By Lemma~\ref{L:tang}, this is the desired perturbation as $\Gamma\cup\wt\gamma$ cannot be polynomially convex. 

For the sake of simplicity, we denote $d\mathscr C(0)$ by the matrix $\mathbf R=(r_{j,k})$, $1\leq j\leq m-1$, $1\leq k,\leq m$. Since $p\in W_2\cap C_1$, by comparing \eqref{E:Web} and \eqref{E:TC1}, we have that the first two columns of $\mathbf R$ are linearly independent. Since the rank of $\mathbf R$ is $m-1$, dropping one of the other columns (not the first two) yields a non-singular $(m-1)\times(m-1)$-minor of $\mathbf R$. Since we can swap any of the $z_2,...,z_{m-1}$ variables via a biholomorphism, we assume without loss of generality that 
\begin{equation}\label{e.*}
\det
\begin{pmatrix}
r_{1,1}& \cdots & r_{1,m-1} \\
\vdots & \ddots & \vdots \\
r_{m-1,1} & \cdots & r_{{m-1},{m-1}}
\end{pmatrix}
\ne 0 .
\end{equation}

\begin{lemma}\label{l.linalg}
Suppose the vector
$(r_{m-1,1},  r_{m-1,2},...,r_{m-1,m-1})$ is non-zero. Then, the following holds. 
\begin{itemize}
\item [(i)] There exists a $j_0\in\{1,...,m-1\}$ such that the $(m-2)\times(m-2)$ submatrix obtained by removing the $j_0$-th and $m$-th columns and the last row of $\bf R$ is nonsingular. 
\item [(ii)] There exists a $k_0 \ne j_0$ in $\{1,...,m-1\}$ such that $(m-2)\times(m-2)$ submatrix obtained by removing the $k_0$-th and $m$-th columns and the 
$(m-2)$-th row of $\bf R$ is nonsingular. 
\end{itemize}
\end{lemma}

\begin{proof} First, we prove (i). After removing the last row and the last column in $\bf R$, we obtain an $(m-2) \times (m-1)$ submatrix, which by~\eqref{e.*}, has full rank. 
Therefore, there exists a column $j_0$ such that after also removing that column, the remaining square minor is nonzero. 

For (ii), observe that the same argument as in (i) shows that there exists a $k'\in\{1,...,m-1\}$ such that the submatrix obtained by removing the $k'$-th and $m$-th columns 
and the 
$(m-2)$-th row of $\bf R$ is nonsingular. If $k' \ne j_0$ we are done. Suppose that $k'=j_0$. Let $T_j\in \R^{m-2}$, $j=1,\dots, m-1$, be the 
columns of the matrix $\bf R$ with the $m$-th column and $(m-2)$-th row removed. By construction, $\{T_j, j\ne k'\}$ are linearly independent, and 
therefore, 
$T_{k'}$ is a linear combination of $\{T_j, j\ne k'\}$. By a small modification of $r^{m-1}_{k'}$, if necessary, we may assume that $T_{k'} \ne 0$, 
and therefore, we may find $k_0\ne k'$ such that $\{T_j, j\ne k_0\}$ are linearly independent, and this proves (ii).
\end{proof}

After a small modification of the coefficients $\theta_j, \kappa_j, \pi_j,\psi_j^s, \varphi_j^s$, $m\leq j\leq n$, 
$2\leq s\leq m-1$, (which can be achieved by a compactly supported $O(3)$-small perturbation of $M$) we assume that the matrix $\bf R$ satisfies Lemma~\ref{l.linalg}. Denote the $j$-th row of $\mathbf R$ by $R^j$, and view it as the vector $(r_{j,1},...,r_{j,m})$ in $\rl^m$.
Let 
$\{\vec e_1,\dots, \vec e_m\}$ be the standard basis in $\R^m$. Recall that the generalized cross product of $R_1,...,R_{m-1}$ is given by 
\bes
\times\mathbf R=R^1\times\cdots\times R^{n-1} = \det 
\begin{pmatrix}
r_{1,1} & \cdots & r_{1,m} \\
\vdots & \ddots & \vdots \\
r_{m-1,1} & \cdots & r_{m-1,m} \\
\vec e_1 & \cdots & \vec e_m 
\end{pmatrix}.
\ees
It is known that $\times \mathbf R$ is orthogonal to the span of all $R^j$'s 

We now replace the entry 
$r^{m-1,m}$ with $r^{m-1,m} + \sigma$, where $\sigma\in\R$ is small. Denote the resulting matrix by ${\bf R}_{\sigma}$. Consider the vector
\bes 
\omega=\times \mathbf R_\sigma-\times\mathbf R=
\det\begin{pmatrix}
r_{1,1}& \cdots & r_{1,m} \\
\vdots & \ddots & \vdots \\
r_{m-1,1} & \cdots & r_{m-1,m} + \sigma \\
\vec e_1 & \cdots & \vec e_m 
\end{pmatrix} - 
\det\begin{pmatrix}
r_{1,1}& \cdots & r_{1,m} \\
\vdots & \ddots & \vdots \\
r_{m-1,1} & \cdots & r_{m-1,m}  \\
\vec e_1 & \cdots & \vec e_m 
\end{pmatrix}.
\ees
Then $\omega_{m}=0$. To compute the other components of $\omega$, denote by $[{\bf R}, j]$, the $(m-1)\times(m-1)$ submatrix 
obtained by removing column $j$ from $\bf R$. Then
$$
\omega_j = \det [{\bf R}, j] - \det [{\bf R}_{\sigma}, j] .
$$
Since the determinant is linear in each column, we obtain that $\omega_j=\sigma \Lambda_j$, where $\Lambda_j$ is, up to a sign, 
the determinant of the submatrix obtained by removing the $j$-th and $m$-th columns and $m-1$-th row from $\bf R$. For example,
\beas
\omega_1&=&
\det\begin{pmatrix}
r_{1,2}& \cdots & r_{1,m} \\
\vdots & \ddots & \vdots \\
r_{m-1,2} & \cdots & r_{m-1,m} + \sigma \\
\vec e_1 & \cdots & \vec e_m 
\end{pmatrix} -\det\begin{pmatrix}
r_{1,2}& \cdots & r_{1,m} \\
\vdots & \ddots & \vdots \\
r_{m-1,2} & \cdots & r_{m-1,m}\\
\vec e_1 & \cdots & \vec e_m 
\end{pmatrix}\\
&=& (-1)^{2m-1}\, \sigma\,
\det\begin{pmatrix}
r_{1,2}& \cdots & r_{1,{m-1}} \\
\vdots & \ddots & \vdots \\
r_{m-2,2} &\cdots & r_{m-2,m-1}
\end{pmatrix}=
(-1)^{2m-1}\, \sigma\,\Lambda_1 .
\eeas
By Lemma~\ref{l.linalg}, there exists $j_0$ such that $\Lambda_{j_0}\ne 0$. This shows that if we modify the matrix $\bf R$ by adding a 
small term $\sigma_{j_0}$ as above, the direction of the tangent vector to $\gamma$ changes, at least in the component $j_0$. 
A similar argument can be used to show that if we add a small $\sigma_{k_0}$ to $r^{m-2,m}$, then the direction of the tangent to $\gamma$ changes in the component $k_0 \ne j_0$. With a  suitable choice of $\sigma_{j_0}$ and $\sigma_{k_0}$ the resulting tangent vectors along with the initial tangent vector to $\gamma$ will be linearly independent.

As a result, we obtain that by a small modification of the entries $(r_{m-2,m}, r_{m-1,m})$, we obtain a $3$-dimensional space of possible tangent vectors to the corresponding $\sigma$ (which is the projection of the corresponding $\gamma$). Note from the expression for $d\mathscr C(0)$ (see \eqref{E:TC1} and \eqref{E:defC1}), that the perturbations of $(r_{m-2,m}, r_{m-1,m})$ can be matched 
by a perturbation of the coefficients of the cubic terms in the coefficients for $M$. This perturbation of coefficients can be achieved by a compactly supported  $C^3$-small perturbation of the initial $f:\M\rightarrow\Cn$. Since $\Pi_{\mathbf x}\ell_p$ is at most $2$-dimensional, we can produce a $\gamma$ this way that is not tangential to $\ell_p$ at $p$, and we are done. 

\bigskip

The proof of the approximation statement in Theorem~\ref{t.main} is similar to that in~\cite{GS1} and \cite{GS2}. 
To prove that $\cont(M)=\mathcal P(M)$, where $\mathcal P(M)$ is the uniform closure of holomorphic polynomials restricted to $M$, we use the following result due to O'Farrel-Preskenis-Walsch, see~\cite{OFPrWa84} or \cite[\S 6.3]{St}. {\em Let $X$ be a compact holomorphically convex set in $\Cn$, and let $X_0$ be a closed subset of $X$ for which $X\setminus X_0$ is a totally real subset of the manifold $\C^n\setminus X_0$. A function $f\in \cont(X)$ can be approximated uniformly on $X$ by functions holomorphic on a neighbourhood of $X$ if and only if $f|_{X_0}$ can be approximated uniformly on $X_0$ by functions holomorphic on $X$.} 

First, we apply the above result (or \cite{HW}) to $X=S$ and $X_0=\emptyset$ to obtain that any $f\in\cont(S)$ can be approximated 
uniformly on $S$ by functions holomorphic on a neighbourhood of $S$. Since by construction $S$ is itself polynomially convex, the Oka-Weil theorem allows us to conclude that $\cont(S)=\mathcal P(S)$. It follows that $X:=M$ and $X_0= S$ satisfy the hypothesis of the O'Farrel-Preskenis-Walsch result, and any $f\in\cont(M)$ can be approximated uniformly on $M$ 
by functions holomorphic on a neighbourhood of $M$. Once again, since $M$ is polynomially convex, by the Oka-Weil theorem,
$\mathcal \cont(M)=\mathcal P(M)$. This proves the theorem.

\end{document}